\documentclass[11pt,reqno]{amsart}
\usepackage{graphicx}
\usepackage[draft]{hyperref}
\usepackage{mathrsfs,array,amssymb,pb-diagram}
\usepackage{amsmath,amsopn,amsfonts,stmaryrd,amsthm}
\usepackage[toc,page,title,titletoc,header]{appendix}
\usepackage{multirow,float}
\usepackage{xcolor}
\usepackage{mathtools}
\usepackage{color}
\usepackage{enumitem}
\setitemize{leftmargin=*}
\setlist[enumerate]{leftmargin=*,label=\rm{(\arabic*)}}
\usepackage[framemethod=TikZ]{mdframed}
\usepackage{bbm}
\usepackage{booktabs}
\usepackage{caption}
\usepackage{bm}
\usepackage{tensor}
\usepackage{cleveref}
\usepackage{comment}

\mathtoolsset{showonlyrefs}
\theoremstyle{plain}
\newtheorem{thm}{Theorem}[section]

\newtheorem{prop}[thm]{Proposition}

\newtheorem*{remark}{Remark}
\newtheorem*{remarks}{Remarks}
\newtheorem{remark*}[thm]{Remark}

\newtheorem{Theorem}[thm]{Theorem}
\newtheorem{Corollary}[thm]{Corollary}
\newtheorem{Lemma}[thm]{Lemma}
\newtheorem{Proposition}[thm]{Proposition}
\newtheorem{theorem}[thm]{Theorem}
\newtheorem*{theorem*}{Theorem}

\theoremstyle{definition}

\newtheorem*{Example}{Example}

\newtheorem*{Definition*}{Definition}

\numberwithin{equation}{section}

\newcommand{\rad}{\operatorname{rad}}

\def\H{\mathbb H}

\def\GL{\mathrm{GL}}

\def\N{\mathbb N}
\def\C{\mathbb C}

\def\R{\mathbb R}
\def\Z{\mathbb Z}
\def\sM{M}
\def\sH{\mathscr{H}}
\def\sS{S}
\def\mM{\mathscr{M}}

\def\re{\mathrm{Re}}

\def\JS#1#2{\left(\frac{#1}{#2}\right)}
\def\JnoS#1#2{(\frac{#1}{#2})}
\def\SL{\mathrm{SL}}

\def\M#1#2#3#4{\begin{pmatrix}#1&#2\\#3&#4\end{pmatrix}}
\def\SM#1#2#3#4{\left(\begin{smallmatrix}#1&#2\\#3&#4\end{smallmatrix}
  \right)}
\def\fg{\mathcal{G}}

\renewcommand{\pmod}[1]{\ \left( \mathrm{mod} \, #1 \right)}
\newcommand{\Pmod}[1]{\ ( \mathrm{mod} \, #1 )}
\newcommand{\sqpt}{l}
\renewcommand{\SS}{\mathcal{L}}

\newcommand{\deltabeta}{\beta}
\newif\ifdefs

\defsfalse

\newif\ifdiscs 
\discstrue

\makeatletter
\@namedef{subjclassname@2020}{%
	\textup{2020} Mathematics Subject Classification}
\makeatother

\ifdiscs
\else
\excludecomment{discussion}
\fi

\ifdefs
\else
\excludecomment{extradetails}
\fi

\allowdisplaybreaks

\setlist[itemize]{noitemsep, topsep=0pt}

\title[Hurwitz class number relations]{Hurwitz class number relations and mock modular forms}
\author{Kathrin Bringmann}
\address{University of Cologne, Department of Mathematics and Computer Science, Weyertal 86-90, 50931 Cologne, Germany}
\email{kbringma@math.uni-koeln.de}
\author{Jia-Wei Guo}
\address{Department of Mathematics, Soochow University, Taipei, Taiwan}
\email{jiaweiguo312@gmail.com}
\author{Ben Kane}
\address{Department of Mathematics, The University of Hong Kong, Pokfulam, Hong Kong}
\email{bkane@hku.hk}
\author{Michael Mertens}
\address{Lehrstuhl f\"ur Algebra und Zahlentheorie, RWTH Aachen University, Pontdriesch 14/16, D-52062 Aachen, Germany}
\email{michael.helmut.mertens@rwth-aachen.de}
\author{Yifan Yang}
\address{Department of Mathematics, National Taiwan University  and National Center for Theoretical Sciences, Taipei 10617, Taiwan}
\email{yangyifan@ntu.edu.tw}
\begin{document}
\date{\today}
\begin{abstract}
A classical class number relation of Hurwitz expresses the Fourier coefficients of the product of a unary theta function and the class number generating function. Here, we establish an infinite family of analogous class number relations obtained by replacing the unary quadratic form $m^2$ by positive-definite binary quadratic forms. These identities involve Cohen's generalized class numbers and depend only on the genus of the underlying quadratic form. For this, we construct a genus-dependent level-lowering operator.
\end{abstract}
\maketitle

\section{Introduction and statement of results}

For $n\in\N_0$ with $n\equiv 0,3\Pmod{4}$, let $H(n)$ denote the Hurwitz class numbers.\footnote{See Section \ref{sec:prelims} for formal definitions and unexplained notation.} By a classical class-number relation of Hurwitz \cite[p. 166]{Hurwitz},\footnote{See also \cite[p. 154]{Eichler}.} for a prime $p$, we have  
\begin{equation}\label{eqn:Eichler}
\sum_{\substack{m\in\Z\\ |m|\leq 2\sqrt{p}}} H\left(4p-m^2\right) =2p.
\end{equation}
This formula  can be interpreted as an identity between Fourier coefficients of modular generating functions. It is natural to ask whether similar relations arise from more general theta functions. We investigate extensions of \eqref{eqn:Eichler}, where the unary quadratic form $m^2$ is replaced by a positive-definite quadratic form. 

For a discriminant $-\delta<0$, we denote by $\mathcal{Q}_{-\delta}$ the set of positive-definite integral binary quadratic forms of discriminant $-\delta$. Our generalizations of \eqref{eqn:Eichler} only depend on the genus of $Q$. We first state our results for the principal genus. The resulting identities involve Cohen's generalized class numbers. For a fundamental discriminant $\Delta$ (including $\Delta=1$), let 
  $L(s,\chi_{\Delta}):=\sum_{n\ge1}\frac{\chi_\Delta(n)}{n^s}$ ($\re(s)>1$) denote the {\it Dirichlet $L$-series associated to} $\chi_{\Delta}:=\left(\frac{\Delta}{\cdot}\right)$. Here $(\frac\cdot\cdot)$ is the extended Legendre symbol. It is well-known that $L(s,\chi_{\Delta})$ has a meromorphic continuation to the entire $s$-plane. Cohen \cite{Cohen} introduced generalized class numbers $H(r,n)$ by\footnote{See the second displayed formula on \cite[p. 273]{Cohen} for the first case.} 

\[
H(r,n):=\begin{cases} L\left(1-r,\chi_{\Delta}\right)\sum_{d|\sqpt}\mu(d)\chi_{\Delta}(d)d^{r-1}\sigma_{2r-1}\left(\frac{\sqpt}{d}\right) & \text{if } (-1)^rn=\Delta \sqpt^2\text{ and }n\in\N,\\
\zeta(1-2r) &\text{if }n=0,\\
0 &\text{otherwise,}\end{cases}
\]
where $\mu(n)$ denotes the M\"obius function, $\zeta$ is the Riemann zeta function, and $\sigma_k(n):=\sum_{d|n}d^k$ is the {\it $k$-th divisor sum}. These numbers occur naturally in the Fourier coefficients of half-integral weight Eisenstein series.  For $r=1$ they recover the classical Hurwitz class number. 

The following identity describes the principal genus case. 
\begin{theorem}\label{thm:PrincipalGenus}
Let $-\delta<0$ be an odd fundamental discriminant and suppose that $Q\in\mathcal{Q}_{-\delta}$ is in the principal genus. Then, for $n\equiv 1\Pmod{4}$, we have\footnote{We write vectors in bold letters throughout.}
\[
\sum_{\substack{\bm{m}\in\Z^2 \\ \delta n-4Q(\bm{m})\equiv 3\Pmod 8}}\sum_{r|\delta}H\left(\frac{\delta n-4Q(\bm{m})}{r^2}\right)r=-[n]_8H(2,n)\sigma_1(\delta),
\]
where $0\leq [n]_8<8$ satisfies $n\equiv [n]_{8}\pmod{8}$.
\end{theorem}
\begin{remark}
The numbers $H(2,n)$ appearing in Theorem \ref{thm:PrincipalGenus} also have geometric significance, such as in van der Geer's study of Humbert divisors (see \cite[Theorem 8.1]{vanderGeer}).
\end{remark}
It is natural to ask for analogous identities to those in Theorem \ref{thm:PrincipalGenus} for $Q\in\mathcal{Q}_{-\delta}$ not in the principal genus. For an odd discriminant $-\delta<0$, the genera of $\mathcal{Q}_{-\delta}$ are determined by coprime squarefree positive integers $D$ and $N$ with $DN=\delta$ which are constructed using the values of the genus characters. Setting $\rad(\delta):=\prod_{p\mid \delta} p$,  this motivates the definition
\[
\mathscr{S}_{\delta}:=\left\{(D,N)\in\N^2: \gcd(D,N)=1,\ D,N\text{ squarefree},\ DN=\rad(\delta)\right\}.
\]
For $(D,N)\in\mathscr{S}_{\delta}$, let $\fg_{-\delta,D,N}\subseteq \mathcal{Q}_{-\delta}$ denote the corresponding genus. Let
\[
\mathscr{S}_{\delta}^{\operatorname{np}}:=\left\{(D,N)\in\mathscr{S}_{\delta}:D>1\right\}.
\]
For $D>1$, our generalizations of \eqref{eqn:Eichler} for $Q\in\fg_{-\delta,D,N}$ also require an extension of the Hurwitz class numbers. We first define a more general class number $h_{D,N}(d)$  for $(D,N)\in\mathscr{S}_{\delta}$ below in \eqref{eq: hDN}. This leads to the generalized Hurwitz class numbers
\begin{equation}\label{eqn:HDNcoeff}
H_{D,N}(n):=\begin{cases}
  \displaystyle-\frac1{12}\prod_{p|D}(p-1)\prod_{p|N}(p+1) &\text{if }n=0, \\
  \displaystyle\sum_{r|\sqpt}\frac{h_{D,N}\left(r^2\Delta\right)}{w_{r^2\Delta}} &\text{if }-n=\sqpt^2\Delta, \\
  0 &\text{otherwise}. \end{cases}
\end{equation}
Here $w_d:=1$ if $d\notin\{-3,-4\}$, $w_{-3}:=2$, and $w_{-4}=3$.
The case $D=N=1$ is the  Hurwitz class number $H(n)$. If $D=1$, then we omit the subscript in $H_{D,N}(n)$. We let $\omega(D)$ denote the number of distinct prime divisors of $D$. 
We obtain the following extension\footnote{
If $-\delta$ is an odd fundamental discriminant, then $H_{1,\delta}(0)=-\frac{1}{12}\sigma_1(\delta)$.
}  of Theorem \ref{thm:PrincipalGenus} for $Q\in\mathcal{Q}_{-\delta}$ not contained in the principal genus.
\begin{theorem}\label{thm:D>1}
 Let $\delta\in\N$ be squarefree such that $\delta\equiv 3\Pmod 4$ and suppose that $(D,N)\in\mathscr{S}_{\delta}^{\operatorname{np}}$.  Then, for $Q\in\fg_{-\delta,D,N}$ and $n\equiv0,1\Pmod 4$, we have
   $$
  \sum_{d|\delta}(-1)^{\omega(\gcd(D,d))}d\sum_{\bm{m}\in\Z^2} H\left(\frac{\delta n-4Q(\bm{m})}{d^2}\right)=120H(2,n)H_{D,N}(0).
  $$
\end{theorem}

The proofs of Theorems \ref{thm:PrincipalGenus} and \ref{thm:D>1} rely on modular properties of the corresponding generating functions. Define ($q:=e^{2\pi i\tau}$ throughout)
\begin{equation*}
\mathcal H_{D,N}(\tau):=\sum_{n\ge0} H_{D,N}(n)q^n.
\end{equation*}
If $D=N=1$, then the generating function
\begin{equation*} 
\mathcal H(\tau):=\mathcal{H}_1(\tau)=\sum_{n\ge0} H(n)q^n
\end{equation*}
is not quite modular. Zagier \cite{Zagier} proved that it has a natural non-holomorphic ``completion'' which satisfies weight $\frac{3}{2}$ modularity on $\Gamma_0(4)$. The \begin{it}Jacobi theta function\end{it} $\Theta(\tau):=\sum_{n\in\Z} q^{n^2}$ is a weight $\frac{1}{2}$ modular form. Noting that the left-hand side of \eqref{eqn:Eichler} is the $4p$-th Fourier coefficient of $\mathcal{H}\Theta$, one can give a modular proof of \eqref{eqn:Eichler}.  For $Q\in\mathcal{Q}_{-\delta}$, one similarly associates a theta function  to $Q$
  $$
  \Theta_Q(\tau):=\sum_{\bm{n}\in\Z^2}q^{Q(\bm{n})}.
  $$
It is well-known that $\Theta_Q$ is a modular form of weight $1$ (see \cite[Proposition 2.1]{Shimura-correspondence}). Recall next the{ \it $U \!$- and $V \! $-operators}, defined for  $f(\tau)=\sum_{n\gg -\infty}c(n)q^n$ and $d\in \mathbb{N}$ by $f|U_d(\tau):=\sum_{n\gg -\infty} c(dn)q^n$ and $f|V_d(\tau):=\sum_{n\gg -\infty} c(n)q^{dn}$, respectively. Since these operators preserve modularity, in the next theorem the left-hand side of Theorem \ref{thm:D>1} arises as the $n$-th Fourier coefficient of a certain modular object. On the other hand, as noted in \cite[Theorem 3.1, Corollary 3.2]{Cohen}, $H(2,n)$ is the $n$-th Fourier coefficient of the following Eisenstein series of weight $\frac{5}{2}$
  \begin{equation*} 
  \mathcal H_{\frac{5}{2}}(\tau):=\frac{1}{120}\Theta^5(\tau)-\frac{1}{6}\Theta(\tau)\frac{\eta^8(4\tau)}{\eta^4(2\tau)},
  \end{equation*}
with $\eta(\tau):=q^{\frac{1}{24}}\prod_{n\ge1} (1-q^n)$ the \begin{it}Dedekind $\eta$-function\end{it}. Reformulating Theorem \ref{thm:D>1} in terms of generating functions yields the following modular identity. 
\begin{Theorem} \label{theorem: D>1}
  Let $\delta\in\N$ be squarefree with $\delta\equiv 3\Pmod 4$ and suppose that  $(D,N)\in\mathscr{S}_{\delta}^{\operatorname{np}}$. Then, for $Q\in \fg_{-\delta,D,N}$, we have
  \begin{equation*} 
  \sum_{d|\delta}\big( \mathcal H_{D,N}|U_d \cdot \Theta_Q|V_{4} U_d\big)|U_{\frac{\delta}{d}} = 120\cdot 2^{\omega(\delta)}H_{D,N}(0)\mathcal H_{\frac{5}{2}}.
  \end{equation*}
\end{Theorem}
To establish Theorem \ref{thm:D>1}, we generalize the modularity of $\mathcal{H}$ in Proposition \ref{proposition: HDN modular} to $\mathcal H_{D,N}$. For $(1,\delta)\in \mathscr{S}_{\delta}$ and $Q\in\fg_{-\delta,1,\delta}$, the left-hand side of \Cref{theorem: D>1} is no longer a modular form. To prove Theorem \ref{thm:PrincipalGenus}, we show that replacing $\Theta_Q$ by a suitable modification of the  theta function and then restricting to certain residue classes of exponents yields a modular form. To state our result, for $f(\tau)=\sum_{n\gg -\infty} c(n) q^n$, we require the {\it sieving operator} $(M,m\in\N)$ $f|S_{M,m}(\tau):=\!\smash{\underset{n\equiv m\Pmod{M}}{\sum}}\! c(n)q^n$. For $Q\in\mathcal{Q}_{-\delta}$, we let
\[
\Theta_{Q,M,m}:=\Theta_Q|S_{M,m}
\]
and define the functions
  \begin{align*}
  	G_{Q,1}:&=\begin{cases}
    \displaystyle
    \sum_{d|\delta}\Big(\mathcal H_\delta|U_d
      \cdot \Theta_{Q,2,0}|V_4 U_d\Big)
      |U_{\frac{\delta}{d}}S_{8,1} &\text{if }\delta\equiv 3\Pmod 8, \\
    \displaystyle
    \sum_{d|\delta}\Big(\mathcal H_\delta|U_d
      \cdot \Theta_{Q,2,1}|V_4 U_d\Big)
      |U_{\frac{\delta}{d}}S_{8,1} &\text{if }\delta\equiv 7\Pmod 8,
    \end{cases}\\
  G_{Q,5}:&=\begin{cases}
    \displaystyle
    \sum_{d|\delta}\Big(\mathcal H_\delta|U_d
      \cdot\Theta_{Q,2,1}|V_4 U_d\Big)
      |U_{\frac{\delta}{d}}S_{8,5} &\text{if }\delta\equiv 3\Pmod 8, \\
    \displaystyle
    \sum_{d|\delta}\Big(\mathcal H_\delta|U_d
      \cdot \Theta_{Q,2,0}|V_4 U_d\Big)
      |U_{\frac{\delta}{d}}S_{8,5} &\text{if }\delta\equiv 7\Pmod 8.
    \end{cases}
  \end{align*}
\begin{Theorem} \label{theorem: principal genus}
  Let $N\in\N$ be squarefree with $N\equiv 3\Pmod 4$. For $Q\in \fg_{-N,1,N}$, $G_{Q,1}$ and $G_{Q,5}$ are modular forms in Kohnen's plus-space $\sM_{\frac{5}{2}}^+(\Gamma_0(16))$.  More precisely, we have
  $$
   G_{Q,j}=-[j]_8 2^{\omega(N)}\sigma_1(N)\mathcal H_{\frac{5}{2}}|S_{8,j}.
  $$
\end{Theorem}

\begin{remark}
	Theorems \ref{theorem: D>1} and \ref{theorem: principal genus} reflect a more general phenomenon described in Theorem \ref{theorem: level reduction}. For example, one may use Theorem \ref{theorem: level reduction} to generalize \eqref{eqn:Eichler} to class number relations coming from theta functions with spherical polynomials (see Section \ref{sec:further}).
\end{remark}

The paper is organized as follows. In Section \ref{sec:prelims}, we recall some preliminaries about modular forms and mock modular forms, as well as operators acting on them, and introduce Hurwitz class numbers and genus theory for binary quadratic forms. In Section \ref{sec:HDNmodular}, we prove modular properties for the generating functions $\mathcal{H}_{D,N}$. Motivated by vanishing properties of their Fourier coefficients, we then define a space of modular objects related to genera of binary quadratic forms in Section \ref{sec:GenusSubspace} and construct a level-lowering operator on this subspace in Section \ref{sec:mainthm}. We prove Theorems \ref{thm:D>1} and \ref{theorem: D>1} in Section \ref{sec:MainD>1} and Theorems \ref{thm:PrincipalGenus} and \ref{theorem: principal genus} in Section \ref{sec:MainPrincipal}. We conclude the paper by discussing further identities obtained by these methods in Section~\ref{sec:further}.

\section*{Acknowledgements}
The authors thank F. Villegas-Rodriguez, who studied the case $D=1$ using intersection theory, for suggesting the class number relations given in this paper. 
The first author has received funding from the European Research Council (ERC) under the European Union’s Horizon 2020 research and innovation programme (grant agreement No. 101001179). The research of the third author was supported by grants from the Research Grants Council of the Hong Kong SAR, China (project numbers HKU 17314122, HKU 17305923).

\section{Preliminaries}\label{sec:prelims}

\subsection{Modular forms}\label{sec:ModularForms}
We adapt Kohnen's definition \cite[Section 1c)]{Kohnen-newform} of modular forms of half-integral weight to simultaneously treat integral and half-integral weights. We let $\mathfrak G_{\kappa}$ be the group of pairs $(\gamma,\phi)$, where $\gamma=\begin{psmallmatrix}
a&b\\c&d
\end{psmallmatrix}\in\GL_2^+(\R)$ and $\phi$ is a holomorphic function on $\H$ satisfying
$$
|\phi(\tau)|=\det (\gamma)^{-\frac{\kappa}{2}}|c\tau+d|^{\kappa}
$$
with group law defined by
$$
  \left(\gamma_1,\phi_1(\tau)\right)  \left(\gamma_2,\phi_2(\tau)\right):=\left(\gamma_1\gamma_2,\phi_1\left(\gamma_2 \tau\right)\phi_2(\tau)\right).
$$
This allows us to define a unified slash operator. For $f:\H\to\C$, we let the group algebra of $\mathfrak G_\kappa$ over $\C$ act on $f$ via (for $(\gamma_j,\phi_j)\in\mathfrak G_\kappa$ and $c_j\in\C$)
$$
f|\sum_{j}c_j\left(\gamma_j,\phi_j\right)(\tau):=\sum_{j}c_j \phi_j^{-1}(\tau)f\left(\gamma_j\tau\right).
$$
To obtain a slash operator depending only on $\Gamma\subseteq\SL_2(\Z)$, we next construct an element of $\mathfrak{G}_{\kappa}$ for each $\gamma\in\Gamma$. Following the standard convention, for odd  \( d \), let $\varepsilon_d := 1 $ if $d \equiv 1 \pmod{4}$ and $\varepsilon_d := i$ if $d \equiv 3 \pmod{4}$. For \( \gamma=\begin{psmallmatrix} a & b \\ c & d \end{psmallmatrix} \in \SL_2(\Z) \) and \( \kappa \in \tfrac{1}{2}\Z \) ($\gamma\in\Gamma_0(4)$ if $\kappa\notin\Z$), set\footnote{We suppress the dependence on \(\kappa\) in the notation, since the weight is always clear from the context.}
\begin{equation}\label{eqn:automorphy}
\phi_\gamma(\tau):= \begin{cases}\JS{c}{d}^{\,2\kappa}\
\varepsilon_d^{-2\kappa}\,
(c\tau+d)^{\kappa} \quad&\text{if}\  \gamma\in \Gamma_0(4),\\
(c\tau+d)^\kappa &\text{if}\  \gamma\notin \Gamma_0(4)\ \text{and}\ \kappa\in\Z.
\end{cases}
\end{equation}
We then let $\gamma^{*} :=(\gamma,\phi_\gamma)\in \mathfrak{G}_{\kappa}$. Hence $\Gamma_0(N)$ (with $4\mid N$ if $\kappa\notin\Z$) may be realized as a subgroup of $\mathfrak G_{\kappa}$ via $\gamma\mapsto\gamma^*=(\gamma,\phi_\gamma)$,  for $\phi_\gamma$ defined in \eqref{eqn:automorphy}.

 Let \( f : \H \to \C \). Since $\gamma^*$ is uniquely determined by $\gamma$,  we define the \emph{weight \( \kappa \) slash operator} by 
\[
f|_{\kappa}\gamma(\tau):=f|\gamma^*(\tau)=\phi_\gamma^{-1}(\tau)f(\gamma\tau).
\]
Let \( \chi \) be a Dirichlet character $\pmod{N}$.  Then $\chi$ may be regarded as a character $\Psi_\chi$ on $\Gamma_0(N)$ by letting $\Psi_\chi\left(\begin{smallmatrix}a&b\\c&d\end{smallmatrix}\right)=\chi(d)$. 
We say that \( f \) is \emph{ modular of weight \( \kappa \)  on \( \Gamma_0(N) \) with character \( \chi \)} if
\begin{equation}\label{eqn:slashconvertnew}
f|_\kappa\gamma = 
\begin{cases}
\varepsilon^{2\kappa}_d\chi(d)f&\text{if }\kappa\in\Z\ \text{and}\ 4\mid N,\\
\chi(d)f&\text{otherwise},\\
\end{cases}
\end{equation}
for every $\gamma\in\Gamma_0(N)$. A holomorphic function $f: \mathbb{H} \rightarrow \mathbb{C} $ is a \emph{(holomorphic) modular form of weight \( \kappa \) on \( \Gamma_0(N) \) with Nebentypus character \( \chi \)} if the following conditions hold:
\begin{enumerate}[label=(\arabic*), leftmargin=*]

\item \( f \) satisfies the transformation law \eqref{eqn:slashconvertnew},
\item $f$ is bounded at the cusps. 
\end{enumerate}

\noindent A holomorphic modular form \( f \) is called a \emph{cusp form} if it vanishes at the cusps. We denote by
$
\sM_{\kappa}(\Gamma_0(N),\chi)
$
and
$
\sS_{\kappa}(\Gamma_0(N),\chi)
$
the spaces of modular forms and cusp forms of weight \( \kappa \) and character \( \chi \) on \( \Gamma_0(N) \), respectively.
We suppress the dependence on \( \chi \) (resp.\ \( N \)) if \( \chi \) is trivial (resp.\ \( N=1 \)). By translation invariance, it has a Fourier expansion of the form
$
f(\tau)=\sum_{n\in\Z} c(n)q^n.
$
For use below, we also recall Kohnen's plus space. If \( \kappa \in \Z+\tfrac{1}{2} \), then \emph{Kohnen's plus space} consists of those modular forms whose Fourier coefficients \( c(n) \) vanish unless $(-1)^{\kappa-\frac12} n \equiv 0,1 \pmod{4}.$

\subsection{Mock modular forms and harmonic Maass forms}\label{sec:mockHarmonic}

Let $\kappa\in\frac12\Z$, $N\in\N$, and $\chi$ be a character (mod $N$) (with $4\mid N$ if $\kappa\notin\Z$). A real-analytic function $f:\H\to \C$ is \begin{it} modular of weight $\kappa$ on $\Gamma_0(N)$ with character $\chi$\end{it} if \eqref{eqn:slashconvertnew} holds for every $\gamma\in\Gamma_0(N)$. It has \begin{it}moderate growth at the cusps\end{it} if $f(\tau)=O(\tau_2^r)$ as $\tau_2\to\infty$ for some $r\geq 0$ (throughout $\tau=\tau_1+i\tau_2$) and analogous growth conditions hold at the other cusps. We let $\mM_{\kappa}(\Gamma_0(N),\chi)$ be the space of real-analytic modular forms of weight $\kappa$ and character $\chi$ on $\Gamma_0(N)$. By translation-invariance, $f\in\mM_{\kappa}(\Gamma_0(N),\chi)$ has a Fourier expansion of the shape 
$
f(\tau)=\sum_{n\in\Z} c_{\tau_2}(n)q^n.
$
A weight $\kappa$ real-analytic modular form $f$ is a \begin{it}harmonic Maass form\end{it} if it satisfies the following additional conditions: 
\begin{enumerate}
\item The function $f$ is annihilated by the \emph{weight $\kappa$ hyperbolic Laplace operator }
\[
\Delta_{\kappa}:=- \tau_2^2\left(\frac{\partial^2}{\partial \tau_1^2}+\frac{\partial}{\partial \tau_2^2}\right) + i\kappa \tau_2\left(\frac{\partial}{\partial\tau_1}+i\frac{\partial}{\partial\tau_2}\right).
\]
\item There exists a polynomial $P$ and $c^-(0)\in\R$ such that $f(\tau) = P(q^{-1})+ c^-(0) \tau_2^{1-\kappa}+O(1)$ as $\tau_2\to\infty$. Analogous conditions are required at the other cusps.
\end{enumerate}
Let $\sH_{\kappa}(\Gamma_0(N),\chi)$ be the space of harmonic Maass forms of weight $\kappa$ and character $\chi$ on $\Gamma_0(N)$. 
The Fourier expansion of a harmonic Maass form $f\in \sH_{\kappa}(\Gamma_0(N),\chi)$ splits as
\[
f(\tau)=f^+(\tau)+ f^{-}(\tau),
\]
where
$
f^+(\tau):=\sum_{n\gg -\infty} c^{+}(n) q^n 
$
is the \begin{it}holomorphic part\end{it} of $f$ and 
\[
f^-(\tau):=c^{-}(0)\tau_2^{1-\kappa}+ \sum_{n<0}c^{-}(n) \Gamma\left(1-\kappa,-4\pi n\tau_2\right) q^{n}
\]
is the \begin{it}non-holomorphic part\end{it} of $f$. Here, for $y>0$, $\Gamma(s,y):=\int_{y}^{\infty} t^{s-1}e^{-t}dt$ is the \begin{it}incomplete gamma function\end{it}.  A \begin{it}mock modular form\end{it} is the holomorphic part of some\ harmonic Maass form. A classical example is the {\it quasimodular Eisenstein series of weight two}
\[
E_2(\tau):=1+24\sum_{n\ge1} \sigma_1(n)q^n.
\]
 Namely, we have
\[
\widehat E_2(\tau):=E_2(\tau)-\frac{3}{\pi \tau_2}\in \sH_2.
\]

\subsection{Operators on (non-holomorphic) modular forms}\label{sec:operators} 
 For a translation-invariant function $f(\tau)=\sum_{n\in\Z}$ $c_{\tau_2}(n)q^n$ and $d\in\N$, we let $U_d$ and $V_d$ be the classical $U$-and $V$ {\it operators }
$$
f| U_d(\tau):=\sum_{n\in\Z}c_{\frac{\tau_2}{d}} (dn)q^n, \qquad
f| V_d(\tau):=\sum_{n\in\Z}c_{d\tau_2}(n)q^{dn}.
$$
Combining these, for $f\in \sH_{\frac{3}{2}}(\Gamma_0(4N))$ and a prime $p$ coprime to $N$, we define the {\it Hecke operator}
\[
f|T_{p^2}(\tau):=\sum_{n\in\Z} \left( c_{\frac{\tau_2}{p^2}}\left(p^2n\right) +\JS{-n}{p} c_{\tau_2}(n) +p c_{p^2\tau_2}\left(\frac{n}{p^2}\right) \right)q^n,
\]
where we set $c_{\tau_2}(x):=0$ if $x\notin\Z$. The following lemma is used repeatedly. 
\begin{Lemma}\label{lem:Up2Tp2rewrite}
We have 
\begin{align*}
  f|U_{p^2}(\tau)&=\frac1{p^2}\sum_{j=0}^{p^2-1}f\left(\frac{\tau+j}{p^2}\right), \\
  f|T_{p^2}(\tau)&=\frac1{p^2}\sum_{j=0}^{p^2-1}f\left(\frac{\tau+j}{p^2}\right)
  +\frac{\varepsilon_p^3}{\sqrt p}\sum_{j=1}^{p-1}\JS{-j}pf\left(\tau+\frac jp\right)+pf\left(p^2\tau\right).
\end{align*}
\end{Lemma}
\begin{proof}
The identity for $U_{p^2}$ follows by orthogonality of roots of unity. Recall that
\begin{equation}\label{eqn:GaussSum}
	\sum_{j\Pmod{p}}\left(\frac{j}{p}\right)e^{\frac{2\pi i a j}{p}} =\varepsilon_p\left(\frac{a}{p}\right) \sqrt{p}
\end{equation} for $a\in \mathbb{Z}$,  $p$ an odd prime. Applying \eqref{eqn:GaussSum} yields the identity for $T_{p^2}$.
\end{proof}

\noindent Modular properties of $f|T_{p^2}$ and $f|U_{p^2}$ follow as in \cite[Theorem 1.7, Proposition 1.5]{Shimura-correspondence} (see also \cite[Propositions 3.4 and 3.7]{OnoBook}). In particular these operators preserve spaces of harmonic Maass forms. 
\begin{Lemma}\label[Lemma]{lem:Tp2Up2}
Suppose that $f\in \sH_{\frac{3}{2}}(\Gamma_0(4N))$ and $p$ is coprime to $N$. Then $f|T_{p^2}\in \sH_{\frac{3}{2}}(\Gamma_0(4N))$ and $f|U_{p^2}\in \sH_{\frac{3}{2}}(\Gamma_0(4pN))$.
\end{Lemma}
Following the proof of \cite[Lemma 5]{WinnieLi}, we obtain a criterion for modular forms whose Fourier expansions are supported on multiples of $p$.
\begin{Lemma}\label[Lemma]{lem:Vpvanish}
Let $k$, $N\in\N$. Suppose that there exists a prime $p\mid N$ with $p^2\nmid N$ and  $f(\tau)=\sum_{n\ge 0} c_{\tau_2}(n)q^{pn}\in \mM_k(\Gamma_0(4N),\chi_{(-1)^k4N})$. Then $f=0$.
\end{Lemma}
Finally, we recall for $M\in\N$ and $m\in\Z$ the {\it sieving operator}
$$
f| S_{M,m}(\tau):=\sum_{n\equiv m\Pmod M}c_{\tau_2}(n)q^n.
$$

 \noindent We also need the following observation about $\mathcal{H}_\frac{5}{2}$.
\begin{Lemma}\label[Lemma]{lem:HSn8-notcusp}
For $m\in\{1,5\}$, the function $\mathcal{H}_{\frac{5}{2}}|S_{8,m}$ is not a cusp form.
\end{Lemma}
\begin{proof}
For an odd prime $p$ and $m\in \{1,5\}$,
\begin{equation}\label{expformular}
	H\left(2,mp^2\right)=L\left(-1,\chi_{m}\right)\left(p^3+1-\chi_m(p)p\right)
\end{equation}
is the $mp^2$-th Fourier coefficient of $\mathcal{H}_{\frac{5}{2}}|S_{8,m}$. Using \eqref{expformular}, we obtain the estimate
\[
\left|H\left(2,mp^2\right)\right| \gg_m p^3=\left(p^2\right)^{\frac{3}{2}}. 
\]
Now, if $\mathcal{H}_{\frac{5}{2}}|S_{8,m}$ were a cusp form, then the standard bound (see \cite[(1.13)]{Shimura-correspondence}) would imply that
\[
\left|H\left(2,mp^2\right)\right| \ll_m \left(p^2\right)^{\frac{5}{4}}.
\]
This is a contradiction. Thus $\mathcal{H}_{\frac{5}{2}}|S_{8,m}$ is not a cusp form and the lemma holds.
\end{proof}

\subsection{Hurwitz class numbers and their generalizations}\label{sec:Hurwitz}

We next recall Hurwitz class numbers and their generalizations. 
For a discriminant $d<0$, we let $h(d)$ denote the class number of primitive integral binary quadratic forms of discriminant $d$. Weighting the quadratic forms by the reciprocal of the size of their automorphism groups in $\operatorname{PSL}_2(\Z)$ yields the modified class numbers 
\begin{equation}\label{eqn:h*def}
   h^*(d):=\frac{h(d)}{w_d}.
\end{equation}
 These satisfy the well-known relation (see \cite[Corollary 7.28]{Cox})
  \begin{equation}\label{eq: h(p^2d)}
  h^*\left(p^2d\right)=\left(p-\JS dp\right)h^*(d).
  \end{equation}
For $n\in\N$, the \begin{it}Hurwitz class number\end{it} $H(n)$  counts the class number of (not necessarily primitive) $Q\in\mathcal{Q}_{-n}$ with the weighting from \eqref{eqn:h*def}. We set $H(0):=-\frac{1}{12}$ and $H(n):=0$ if $n\in\N_0$ with $n\not\equiv 0,3\Pmod{4}$. Write $n\equiv 0,3\pmod{4}$ as $-n=\sqpt^2\Delta$ with $\Delta<0$ a fundamental discriminant. Then the Hurwitz class numbers are related to ordinary class numbers by  
\begin{equation*}
H(n)=\displaystyle\sum_{r|\sqpt}\frac{h\left(r^2\Delta\right)}{w_{r^2\Delta}}=\displaystyle\sum_{r|\sqpt}h^*\left(r^2\Delta\right).
\end{equation*}
We now introduce the class numbers appearing in Theorem \ref{thm:D>1}. Let $D$ and $N$ be squarefree and coprime. For $d=\Delta r^2\in-\N$, where $\Delta$ is a fundamental discriminant, we define\footnote{The numbers $h_{D,N}(d)$ admit an interpretation in terms of optimal embeddings into Eichler orders (see \cite{Voight}). }
\begin{equation} \label{eq: hDN}
h_{D,N}(d):=h\left(r^2\Delta\right)\prod_{p|DN}m_{D,N,p}\left(r^2\Delta\right),
\end{equation}
where the local factors $m_{D,N,p}$ are given by
\[
m_{D,N,p}\left(r^2\Delta\right):=\begin{cases}
  \displaystyle
  1-\JS{\Delta}p &\text{if }p|D\text{ and }p\nmid r, \\
  0 &\text{if }p|D\text{ and }p|r, \\
  \displaystyle 1+\JS{\Delta}p &\text{if }p|N\text{ and }p\nmid r, \\
  2 &\text{if }p|N\text{ and }p|r. \end{cases}
\]
Using the same weighting as in \eqref{eqn:h*def}, we define
\[
h_{D,N}^*(d)=\frac{h_{D,N}(d)}{w_d}.
\]

\subsection{Genus theory for binary quadratic forms}\label{sec:genus}

We now review the genus theory of binary quadratic forms, see \cite[Section 2C]{Cox} for details. Let $-\delta\in-\N$ be a fundamental discriminant. For a fundamental discriminant $d\mid\delta$, we define the \begin{it}genus character\end{it} $\chi_d(Q):=(\frac{d}{a})$, with $a$ any integer coprime to $\delta$ represented by $Q$.\footnote{A primitive quadratic form $Q\in \mathcal{Q}_{-\delta}$ primitively represents an integer coprime to $\delta$, and any $Q\in\mathcal{Q}_{-\delta}$ must be primitive because $-\delta$ is fundamental.}  Binary quadratic forms $Q_1$ and $Q_2$ of discriminant $-\delta$ are \textit{in the same} \begin{it}genus\end{it} if $\chi_d(Q_1)=\chi_d(Q_2)$ for all fundamental discriminants $d\mid \delta$. This induces an equivalence relation and hence $\mathcal{Q}_{-\delta}$ is partitioned into genera. Since  $\chi_d(Q_1)=\chi_d(Q_2)$ for $Q_1$, $Q_2$ in the same genus $\fg$,  $\chi_d(\fg):=\chi_d(Q)$ for any $Q\in\fg$ is a well-defined genus invariant. To describe the genera more explicitly, we recall a standard generating set for the genus characters. Uniquely writing
\[
-\delta = \varepsilon_{\delta,2} 2^{\alpha} \prod_{\substack{p\mid \delta\\p\ne2}} {p^*}
\]
for some $\varepsilon_{\delta,2}\in\{\pm 1\}$ and $p^*:=(\frac{-1}{p})p$, we see that the characters $\chi_{p^*}$ and $\chi_{\varepsilon_{\delta,2} 2^{\alpha}}$ generate the group of genus characters of the class group of binary quadratic forms of discriminant $-\delta$. For $Q\in \mathcal{Q}_{-\delta}$, we have 
\begin{equation} \label{eq: character relation}
\chi_{\varepsilon_{\delta,2} 2^{\alpha}}(Q)\prod_{\substack{p\mid \delta\\p\ne2}} \chi_{p^*}(Q)=1.
\end{equation}
 For  $\delta=DN$ with $\gcd(D,N)=1$, we let $\fg_{-\delta,D,N}$ be the genus of binary quadratic forms of discriminant $-4\delta$ determined by 
  \begin{equation}\label{eqn:fgdef}
  \chi_{p^*}\left(\fg_{-\delta,D,N}\right)=\begin{cases}
    1 &\text{if }p|N, \\
    -1 &\text{if }p|D. \end{cases}
  \end{equation}
The \begin{it}principal genus\end{it} is the genus $\fg=\fg_{-\delta,1,\delta}$ satisfying $\chi_d(\fg)=1$ for all fundamental discriminants $d\mid \delta$. We repeatedly use the following criterion, which follows from the fact that we can choose $\bm{m}\in\Z^2$ so that $Q(\bm{n}+p\bm{m})$ is coprime to $\delta$.
\begin{Lemma}\label[Lemma]{lem:genuscharnotrelprime}
Suppose that $Q\in\mathcal{Q}_{-\delta}$ with $-\delta$ a fundamental discriminant and $p\mid \delta$ is odd. If $p\nmid n$ and $Q(\bm{n})=n$ for some $\bm{n}\in\Z^2$, then 
\[
\chi_{p^*}(Q)=\left(\frac{n}{p}\right).
\]
\end{Lemma}

\subsection{The $U_d$- and $W_p$-operators}\label{sec:UpWp}

We now introduce the main operator used throughout. As elements of the group algebra,  we have (see \cite[p. 40]{Kohnen-newform} for the $U$-operator) 
\begin{align*} 
  U_d:f\mapsto
  d^{\frac{\kappa}{2}-1}\sum_{\ell=0}^{d-1}f\Big|\left(
  \M 1\ell0d,d^{\frac{\kappa}{2}}\right)\quad\text{and}\quad
  &V_d:f\mapsto
  d^{-\frac{\kappa}{2}}f\Big|\left(
  \M d001,d^{-\frac{\kappa}{2}}\right).
\end{align*}
Following Kohnen \cite{Kohnen-newform}, we define the operator $W_{\kappa,p}$ which plays a central role below.
\begin{Definition*} 

	Let $\kappa\in\frac12\Z$. For an odd prime divisor $p$ of $N$ such that $p^2\nmid N$, we let\footnote{We omit the weight $\kappa$ from notation if it is clear from the context what the weight is.}
	\begin{equation*} 
		W_{\kappa,p}:=\left(\M pr{4N}{ps},\phi_{\kappa, p}\right)\in \mathfrak{G}_{\kappa}.
	\end{equation*}
Here, for $r,s\in\Z$ satisfying $p^2s-4Nr=p$, we define
\[
\phi_{\kappa, p}(\tau):=\varepsilon_p^{-2\kappa}p^{-\frac{\kappa}{2}}(4N\tau+ps)^{\kappa}.
\]
\end{Definition*}
One checks that for $\kappa\in \Z+\frac{1}{2}$, $W_{\kappa,p}$ acts independent of $r$ and $s$ for $f\in\mM_{\kappa}(\Gamma_0(4N),\chi)$. \hspace{-.1cm}The actions of $W_p^2$ and $W_p U_p$ were considered in \cite[pp. 39 and 42]{Kohnen-newform}.
\begin{Lemma} \label[Lemma]{lemma: Atkin-Lehner}
  Let $p|N$ but $p^2\nmid N$ and $\chi$ be a Dirichlet
  character (mod $4N$). Decompose  
  $\chi=\chi_1\chi_2$ with Dirichlet characters $\chi_{j} $ of modulus $\frac{4N}{p}$
  and $p$, respectively. Let $f\in\mM_{\kappa}(\Gamma_0(4N),\chi)$.
  \begin{enumerate}[label=\rm(\arabic*)]
    \item We have
    \begin{equation*} 
      f|W_{p}^2=\varepsilon_p^{2k+1}\JS{\frac{N}{p}}p
      \overline{\chi_1(p)}\chi_2(-1)f.
    \end{equation*}
    \item Let $d|N$ be coprime to $p$. Then we have
      \begin{equation*} 
        f|W_p U_d=\overline{\chi_2(d)}f|U_d W_p.
      \end{equation*}
  \end{enumerate}
\end{Lemma}

For modular forms of integral weight, analogous formulas hold. Since the proofs are similar to those in \cite[pp. 39]{Kohnen-newform} and [42], we omit them.
\begin{Lemma} \label[Lemma]{lemma:W^2 2}
  Let $p|N$ but $p^2\nmid N$ and $\chi$ be a Dirichlet character. Decompose $\chi=\chi_1\chi_2$ with $ \chi_1, \chi_2$ Dirichlet characters of modulus $\frac{4N}{p}$ and $p$, respectively. Let $f\in\mM_k(\Gamma_0(4N),\chi)$.
  \begin{enumerate}[label=\rm(\arabic*)]
  \item We have
    \begin{equation*}
    f|W_{p}^2=\overline{\chi_1(p)}\chi_2(-1)f.
    \end{equation*}
  \item If $d$ is coprime to $p$, then
    \begin{equation*}
    f|W_p U_d=\overline{\chi_2(d)}f|U_d W_p.
    \end{equation*}
  \end{enumerate}
\end{Lemma}
The proof of the following lemma is similar to that of \cite[Proposition 4]{Kohnen-newform}.
\begin{Lemma} \label[Lemma]{lemma: g eigenvalues}
Assume that $N\in\N$ is odd and squarefree. Let $k\in\N$ and $f\in \mM_k(\Gamma_0(4N),\chi_{(-1)^k4N})$.
\begin{enumerate}[label=\rm(\arabic*)]
\item  
For a prime divisor $p$ of $N$, we have
  \begin{equation*}
  f|\left(p^{-\frac{k}{2}+\frac{1}{2}}U_p W_p\right)^2=\JS{-1}pf.
  \end{equation*}
\item  The following are equivalent: 
\begin{itemize}[leftmargin=*]
 \item[\textup{(a)}] The function  $f(\tau)=\sum_{n\ge0} c_{\tau_2}(n)q^n$ is an eigenfunction  of $p^{-\frac{k}{2}+\frac{1}{2}}U_p W_p$ with eigenvalue $\pm \varepsilon_p$.
 \item[\textup{(b)}] We have $c_{\tau_2}(n)=0$ for all $n$ satisfying $(\frac np)=\mp(\frac{(-1)^{k-1}\frac{N}{p}}p)$.
\end{itemize}
\end{enumerate}
\end{Lemma}

\section{Modular properties of $\mathcal H_{D,N}$}\label{sec:HDNmodular}

Throughout this section, let $D,N\in\N$ be coprime, odd, and squarefree. We investigate the modular properties of $\mathcal{H}_{D,N}$. For $D=N=1$, Zagier (see \cite{Zagier} or \cite{HirzebruchZagier}) showed that 
\begin{equation}\label{eqn:Hhat}
\widehat{\mathcal H}(\tau):=\mathcal H(\tau)+\frac{1}{8\pi\sqrt{\tau_2}} + \frac1{4\sqrt{\pi}}\sum_{n\ge1} n\Gamma\left(-\frac{1}{2},4\pi n^2\tau_2\right) q^{-n^2}\in \sH_{\frac{3}{2}}\left(\Gamma_0(4)\right).
\end{equation}
 To generalize \eqref{eqn:Hhat}, we require the following lemma.

\begin{Lemma} \label[Lemma]{lemma: mock to modular}
Suppose that
  $$
  f(\tau)=\sum_{n\ge0} c^+(n)q^n+\frac{c^-(0)}{\sqrt{\tau_2}}+2\sum_{n\ge1}nc^-(n)\Gamma\left(-\frac{1}{2},4\pi n^2\tau_2\right)q^{-n^2}\in \sH_{\frac{3}{2}}\left(\Gamma_0(4N)\right).
  $$
  Let $p$ be a prime coprime to $2N$ such that $c^-(pn)=c^-(n)$ for all $n\in\N_0$. Define $f_p$ by
  \begin{multline*}
  f_p(\tau):=\sum_{n\ge0}\hspace{-.06cm}\left(pc^+\left(\frac{n}{p^2}\right)
    +\JS{-n}pc^+(n)\right)q^n
  +\frac{c^-(0)}{\sqrt{\tau_2}}
+ 2\sum_{n\ge1}nc^-(n)\Gamma\hspace{-.06cm}\left(-\frac{1}{2},4\pi n^2\tau_2\right)q^{-n^2},
\end{multline*}
where we set $c^+(x):=0$ if $x\notin\Z$. Then $f_p\in \sH_{\frac{3}{2}}\left(\Gamma_0(4pN)\right)$.  Consequently, $f_p-f\in \sM_{\frac{3}{2}}(\Gamma_0(4pN))$ and $f_p+f\in \sH_{\frac{3}{2}}(\Gamma_0(4pN))$.
\end{Lemma}

\begin{proof}
Combining Lemma \ref{lem:Up2Tp2rewrite} and Lemma \ref{lem:Tp2Up2}, we obtain
\begin{equation}\label{eqn:fpharmonic}
	pf\left(p^2\tau\right)+\frac{\varepsilon_p^3}{\sqrt p}\sum_{j=1}^{p-1}\JS{-j}pf\left(\tau+\frac jp\right)=f|T_{p^2}(\tau)- f|U_{p^2}(\tau)\in \sH_{\frac{3}{2}}\left(\Gamma_0(4pN)\right).
\end{equation}
We claim that $f_p$ equals the left-hand side of \eqref{eqn:fpharmonic}, from which all of the claims of the lemma follow. 

Using \eqref{eqn:GaussSum}, one verifies that the holomorphic part of the left-hand side in \eqref{eqn:fpharmonic} is
  \begin{equation*}
  \sum_{n\ge0}\left(pc^+\left(\frac{n}{p^2}\right)+\JS{-n}pc^+(n)\right)q^n,
  \end{equation*}
matching the holomorphic part of $f_p$. It remains to show that $f_p^-$ agrees with the non-holomorphic part of the left-hand side of \eqref{eqn:fpharmonic}. By \eqref{eqn:GaussSum}, we have 
\[
 \frac{\varepsilon_p^3}{\sqrt{p}}\sum_{j=1}^{p-1}
      \JS{-j}p f^{-}\hspace{-.1cm}\left(\tau+\frac jp\right)=2\sum_{\substack{n\ge1}}
      nc^-(n)\Gamma\left(-\frac{1}{2},4\pi n^2\tau_2\right)q^{-n^2}.
\]
Therefore, the non-holomorphic part of the left-hand side of \eqref{eqn:fpharmonic} is
\begin{equation*}
\frac{c^-(0)}{\sqrt{\tau_2}} + 2\sum_{n\ge1}nc^-(n) \Gamma\left(-\frac{1}{2},4\pi n^2\tau_2\right)q^{-n^2}
\end{equation*}
This matches the non-holomorphic part of $f_p$. 
\end{proof}
We next establish a recursion satisfied by $H_{D,N}$.
\begin{Lemma}\label[Lemma]{lem:HDNHecke}
  Let $D$, $N\in N$ be odd, squarefree, coprime, and $p\nmid DN$ be a prime. Then we have
  \begin{equation*}
  pH_{D,N}\left(\frac{n}{p^2}\right)+\JS{-n}pH_{D,N}(n)+H_{D,N}\left(p^2n\right)
  =(p+1)H_{D,N}(n).
  \end{equation*}
\end{Lemma}
\begin{proof}
Using \eqref{eq: h(p^2d)}, a straightforward calculation yields that, for any prime $p\nmid DN$,
\begin{equation} \label{eq: hDN(p^2d)}
  h_{D,N}^*\left(p^2d\right)=\left(p-\JS dp\right)h_{D,N}^*(d).
\end{equation}
Plugging \eqref{eq: hDN(p^2d)} into \eqref{eqn:HDNcoeff}, a standard calculation yields the lemma.
\end{proof}
Lemma \ref{lem:HDNHecke} translates into identities for the generating functions.
\begin{Lemma}\label[Lemma]{lem:HDNtoHDNp}
Let $D$, $N\in\N$ be odd, squarefree, coprime integers and let $p\nmid DN$ be a prime.
\begin{enumerate}
\item We have
\[
 \mathcal{H}_{D,pN}(\tau)= \sum_{n\ge0}\left(pH_{D,N}\left(\frac{n}{p^2}\right)+\left(\JS{-n}p+1\right)H_{D,N}(n)\right)q^n.
\]
\item We have
\[
 \mathcal{H}_{pD,N}(\tau)= -\sum_{n\ge0}\left(pH_{D,N}\left(\frac{n}{p^2}\right)+\left(\JS{-n}p-1\right)H_{D,N}(n)\right) q^n.
\]
\end{enumerate}
\end{Lemma}
\begin{proof}
(1) Using \Cref{lem:HDNHecke}, we have
\begin{equation}\label{eqn:HDNtoHDNp1}
pH_{D,N}\!\left(\frac{n}{p^2}\right)+\left(\JS{-n}p+1\right)\!H_{D,N}(n) = (p+2)H_{D,N}(n) - H_{D,N}\left(p^2n\right).
  \end{equation}
  Write $n=\sqpt^2\delta$, where $-\delta$ is a fundamental discriminant. Then \eqref{eqn:HDNcoeff} yields
\begin{equation}\label{eqn:HHeckerelrewrite}
pH_{D,N}(n)-H_{D,N}\left(p^2n\right) = -\sum_{\substack{r|\sqpt\\ p\nmid r}}\left(1-\left(\frac{-\delta}{p}\right)\right)   h_{D,N}^*\left(-r^2\delta\right).
\end{equation}
 Using \eqref{eqn:HHeckerelrewrite}, \eqref{eqn:HDNcoeff}, and \eqref{eq: hDN}, the right-hand side of \eqref{eqn:HDNtoHDNp1} becomes
\begin{equation*}
\sum_{\substack{r\mid \sqpt\\ p\nmid r}}\left(1+\left(\frac{-\delta}{p}\right)\right) h_{D,N}^*\left(-r^2\delta\right) + 2\sum_{\substack{r\mid \sqpt\\ p\mid r}} h_{D,N}^*\left(-r^2\delta\right)=H_{D,pN}(n).
\end{equation*}

\noindent(2) From \eqref{eq: hDN}, we obtain
\[
h_{pD,N}^*\left(-r^2\delta\right)= h_{D,N}^*\left(-r^2\delta\right) m_{pD,N,p}\left(-r^2\delta\right).
\]
Plugging this into \eqref{eqn:HHeckerelrewrite}, the claim follows.
\end{proof}

Motivated by \eqref{eqn:Hhat}, we define\footnote{Recall that only the case $D=1$ requires a completion.}
\begin{equation}\label{eqn:HnHat}
      \widehat{\mathcal H}_{N}(\tau):=\mathcal H_{N}(\tau)+\frac{2^{\omega(N)-3}}{\pi\sqrt{\tau_2}} + \frac{2^{\omega(N)-2}}{\sqrt{\pi}}   \sum_{n\ge1}n\Gamma\left(-\frac{1}{2},4\pi n^2\tau_2\right)q^{-n^2}.
\end{equation}
We now deduce the modularity properties of $\mathcal{H}_{D,N}$. 
\begin{Proposition}\label[Proposition]{proposition: HDN modular}
  Let $D$, $N\in\N$ be odd, squarefree, and coprime and let $p\nmid DN$ be a prime.
  \begin{enumerate}[label=\rm(\arabic*)]
    \item We have $\widehat{\mathcal H}_{N}\in \sH_{\frac{3}{2}}\left(\Gamma_0(4N)\right)$. 
    \item If $D>1$, then $\mathcal H_{D,N}\in \sM_{\frac{3}{2}}(\Gamma_0(4DN))$.
  \end{enumerate}
\end{Proposition}

\begin{proof}

\noindent(1) We argue by induction on the number of prime divisors of $N$. The case $N=1$ is \eqref{eqn:Hhat}. Now suppose that the claim holds for some squarefree $N$ and let $p\nmid N$ be a prime.  Plugging $D=1$ into Lemma \ref{lem:HDNtoHDNp} (1) and taking $f=\widehat{\mathcal{H}}_{N}\in \sH_{\frac{3}{2}}\left(\Gamma_0(4N)\right)$ by the induction hypothesis, the right-hand side of Lemma \ref{lem:HDNtoHDNp} (1) is the holomorphic part of $f+f_p$, so the modularity of $\widehat{\mathcal{H}}_{pN}$ follows by Lemma \ref{lemma: mock to modular}.

\noindent(2) We fix $N$ and use induction on the number of prime factors $r$ of $D$. For the base case, we take $D=p$ prime. By part (1), for $f=\widehat{\mathcal{H}}_{N}$, we have $f\in \sH_{\frac{3}{2}}\left(\Gamma_0(4N)\right)$.  Lemma \ref{lem:HDNtoHDNp} (2) implies that $\mathcal{H}_{p,N}=f-f_p$, where $f_p$ is defined in Lemma \ref{lemma: mock to modular}. We conclude from Lemma \ref{lemma: mock to modular} that $\mathcal H_{p,N}=f-f_p\in \sM_{\frac{3}{2}}(\Gamma_0(4pN))$. 
Next suppose that $r\in \mathbb{N}$ and assume that $f=\mathcal{H}_{D,N}\in \sM_{\frac{3}{2}}(\Gamma_0(4DN))$ for all $D$ having $r$ prime factors. Then Lemma \ref{lem:HDNtoHDNp} (2) yields $\mathcal{H}_{pD,N}=f-f_p$ and Lemma \ref{lemma: mock to modular} implies that $\mathcal{H}_{pD,N}=f-f_p\in \sM_{\frac{3}{2}}(\Gamma_0(4pDN))$.
\end{proof}

We also require an expression for $\mathcal{H}_{D,N}$ in terms of Hurwitz class numbers. 

\begin{Lemma}\label[Lemma]{lemma: alternative HDN}
For $D,N\in\N$ squarefree and coprime, we have
  \begin{equation*} 
  \mathcal H_{D,N}(\tau)
  =(-1)^{\omega(D)}\sum_{r|DN}r\sum_{n\ge0} b_{D,N,r}(n)H(n)q^{r^2n},
  \end{equation*}
  where
  \begin{equation*} 
    b_{D,N,r}(n):=
    \prod_{p|\gcd\left(D,\frac{DN}{r}\right)}\left(\JS{-n}p-1\right)
    \prod_{p|\gcd\left(N,\frac{DN}{r}\right)}\left(\JS{-n}p+1\right).
  \end{equation*}
\end{Lemma}
\begin{proof}
By comparing Fourier coefficients on both sides, the claim is equivalent to showing that for all squarefree and coprime $D,N\in\N$ and $n\in\N$, we have
\begin{equation}\label{eqn:inductionclaim}
H_{D,N}(n)=(-1)^{\omega(D)}\sum_{r\mid DN}r b_{D,N,r}\left(\frac{n}{r^2}\right) H\left(\frac{n}{r^2}\right).
\end{equation}
We first prove \eqref{eqn:inductionclaim} for $N=1$  by induction on the number of prime divisors of $D$. The claim holds trivially for $D=1$, as $H_{1,1}(n)=H(n)$. Now suppose the claim for $\mathcal{H}_{D,1}$ and let $p\nmid D$ be an odd prime. By Lemma \ref{lem:HDNtoHDNp} (2), we have
  \begin{align}
  \mathcal H_{pD,1}(\tau)&=-p\sum_{n\ge0} H_{D,1}(n)q^{p^2n}
  -\sum_{n\ge0}\left(\JS{-n}p-1\right) H_{D,1}(n)q^n \nonumber \\ 
  &=(-1)^{\omega(pD)}\sum_{r|D}r\sum_{n\ge0} H(n)
  \left(p b_{D,1,r}(n)q^{p^2r^2n}+\left(\JS{-n}p-1\right)b_{D,1,r}(n)q^{r^2n}\right),\label{eqn:HpDinduction}
  \end{align}
using the induction hypothesis \eqref{eqn:inductionclaim}.
  Note that $b_{pD,1,pr}(n)=b_{D,1,r}(n)$. Also, since $\gcd(p,r)=1$, we have $( (\frac{-n}{p})-1)b_{D,1,r}(n)=b_{pD,1,r}(n)$. Plugging these into \eqref{eqn:HpDinduction}, it follows that
  $$
  \mathcal H_{pD,1}(\tau)
  =(-1)^{\omega(pD)}\sum_{r|pD}r\sum_{n\ge0}
  b_{pD,1,r}(n)H(n)q^{r^2n}.
  $$
 This proves \eqref{eqn:inductionclaim} for $H_{D,1}$ with $D$ squarefree by induction.

We next use induction on $N\in\N$ coprime to $D$. The base case $N=1$ is proved above. Suppose the claim holds for $\mathcal{H}_{D,N}$ and let $p\nmid DN$ be an odd prime. We plug the induction hypothesis \eqref{eqn:inductionclaim} into Lemma \ref{lem:HDNtoHDNp} (1), to obtain
  \begin{equation}\label{eqn:HDpNinduction}
      \mathcal H_{D,pN}(\tau)=(-1)^{\omega(D)}\!\sum_{r|DN}\!r\sum_{n\ge0} H(n)
      \left(p b_{D,N,r}(n)q^{p^2r^2n}+\left(\JS{-n}p+1\right)b_{D,N,r}(n)q^{r^2n}\right).
  \end{equation}
Note that $b_{D,pN,pr}(n)=b_{D,N,r}(n)$ and $b_{D,pN,r}(n)=(\JnoS{-n}{p}+1)b_{D,N,r}(n)$. Plugging these into \eqref{eqn:HDpNinduction} yields the claim \eqref{eqn:inductionclaim} for $N\mapsto pN$. 
\end{proof}

Finally we record an auxiliary identity for sums of Hurwitz class numbers similar to  \eqref{eqn:Eichler}.
\begin{Lemma}\label[Lemma]{lem:UnarySum}
\begin{enumerate}[label=\rm(\arabic*)]
\item 
For $n\in\N$ and $j\in\{0,1\}$, we have 
\[
\sum_{\substack{t\in\Z\\ t\equiv j\pmod{2},\\ t^2\equiv n-3\Pmod8}} H\left(n-t^2\right)=
\begin{cases} 
	\frac{1}{12}\sigma(n) & \text{if $n\equiv 3\pmod{4}$ and $j=0$},\\
	\frac{2}{21}\sigma(n) & \text{if $n\equiv 4\pmod{8}$ and $j=1$},\\
	0&\text{otherwise}.
\end{cases}
\]
\item 
For $n\in\N$ and $j\in\{0,1\}$, we have
\[
\sum_{\substack{t\in\Z\\ t\equiv j\pmod{2},\\ 4t^2\equiv n-3\Pmod8}} H\left(n-4t^2\right)=\begin{cases} 
\frac{1}{12}\sigma(n)&\text{if }n\equiv 3+4j\pmod{8},\\
0&\text{otherwise}.
\end{cases}
\]
\end{enumerate}
\end{Lemma}
\begin{proof}
Since the proofs of both parts are similar, we only prove (1). Note first that the generating function of the left-hand side is
\begin{equation}\label{sum}
	\sum_{n\ge0}\sum_{\substack{t\in\Z\\ t\equiv j\pmod{2}\\ t^2\equiv n-3\pmod{8}}} H\left(n-t^2\right) q^n=\left(\mathcal{H}|S_{8,3} \Theta|S_{2,j}\right)(\tau).
\end{equation}
Note that $\widehat{\mathcal{H}}|S_{8,3}=\mathcal{H}|S_{8,3}$, because the non-holomorphic part of $\widehat{\mathcal{H}}$ is only supported on powers of $q$ which are negatives of squares. We conclude from \eqref{eqn:Hhat} that \eqref{sum} is in $M_2\left(\Gamma_0(64)\right)$.The generating function for the right-hand side of (1) is a multiple of $E_2|S_{4,3}$ or $E_2|S_{8,4}$, which is also in $\sM_2(\Gamma_0(64))$. Part (1) then follows by the valence formula after computing $96$ Fourier coefficients.
\end{proof}

\section{Genus subspaces of $\sM_{\kappa}(\Gamma_0(4N))$ and $\sH_{\kappa}(\Gamma_0(4N))$}\label{sec:GenusSubspace}

Throughout this section, we assume that $N\in\N$ is odd and squarefree and define $\delta_N:= 4N$ if $N\equiv 1\pmod{4}$ and $\delta_N:= N$ if $N\equiv 3\pmod{4}$. 
Then $-\delta_N$ is a fundamental discriminant and 
\begin{equation}\label{eqn:-4Nsplit}
-\delta_N=\begin{cases}
\prod_{p\mid N}{p^*}&\text{if }N\equiv 3\pmod{4},\\[-0.3cm]
\\
-4\prod_{p\mid N}{p^*}&\text{if }N\equiv 1\pmod{4}.
\end{cases}
\end{equation}
For a real character $\chi$ (mod $4N$) and $p\mid \delta_N$, we uniquely split $\chi=\varphi\psi$ with $\varphi$ having modulus coprime to $p$ and $\psi$ having modulus equal to a power of $p$. For $n\in\Z$ coprime to $4N$, we use the splitting \eqref{eqn:-4Nsplit} to define for $p\mid\delta_N$
\begin{equation}\label{eqn:epsilonchi1p}
\varepsilon_{N,\chi,n}(p):=\varphi(p)\begin{cases}
 \chi_{-4}(n)&\text{if $p=2$ and $N\equiv 1\pmod{4}$},\\
\chi_{p^*}(n) &\text{if }p\mid N.
\end{cases}
\end{equation}
Thus $\varepsilon_{N,\chi,n}$ is a sign function depending on the prime divisors of $\delta_N$. Since $n$ is coprime to $4N$, $\varepsilon_{N,\chi,n}$ is a function from $\{ p\mid \delta_N\}$ to $\{\pm 1\}$. Define 
\begin{equation}\label{eqn:epsilongenus}
\mathcal{E}_N:=\left\{\varepsilon:\{p\mid \delta_N\}\to \{\pm 1\}\text{ such that }\prod_{p\mid \delta_N} \varepsilon(p) =1\right\}.
\end{equation}
For $\varepsilon\in \mathcal{E}_N$, $\kappa\in\frac{1}{2}\Z$, and a character $\chi$ of modulus $4N$, we define $\mM_{\kappa}^{\varepsilon,\pm}(\Gamma_0(4N),\chi)$ to be the space of functions $f$ satisfying the following conditions:
\begin{enumerate}
\item We have $f\in \mM_{\kappa}(\Gamma_0(4N),\chi)$.
\item Write the Fourier expansion $f(\tau)=\sum_{n\in\Z} c_{\tau_2}(n)q^n $. If $n$ is coprime to $4N$ and 
$\varepsilon_{N,\chi,\pm n}\neq \varepsilon$
(i.e., there exists $p\mid \delta_N$ for which $\varepsilon_{N,\chi,\pm n}(p)\neq \varepsilon(p)$), then $c_{\tau_2}(n)=0$. 
\end{enumerate}
We further let $\sH_{\kappa}^{\varepsilon,\pm}(\Gamma_0(4N),\chi)$ (resp. $\sM_{\kappa}^{\varepsilon,\pm}(\Gamma_0(4N),\chi)$) be the subspace of $\mM_{\kappa}^{\varepsilon,\pm}(\Gamma_0(4N),\chi)$ contained in $\sH_{\kappa}(\Gamma_0(4N),\chi)$ (resp. $\sM_{\kappa}(\Gamma_0(4N),\chi)$). We omit $\chi$ in the notation if it is trivial.
\begin{remarks}
\noindent

\noindent
\begin{enumerate}
\item Noting the relation \eqref{eq: character relation}, for a genus $\fg$ of binary quadratic forms of discriminant $-\delta_N$, the genus characters give rise to a function $\varepsilon_{\fg}\in \mathcal{E}_N$ by $\varepsilon_{\fg}(2):=\chi_{-4}(\fg)$ if $N\equiv 1\pmod{4}$ and $\varepsilon_{\fg}(p):=\chi_{p^*}(\fg)$ for $p\neq 2$. 
Moreover, since $\#\mathcal{E}_N=2^{\omega\left(\delta_N\right)-1}$ equals the number of genera computed by Gauss, every $\varepsilon\in \mathcal{E}_N$ arises from a unique genus $\fg$ of binary quadratic forms for which $\varepsilon=\varepsilon_{\fg}$. We refer to $\sM_{\kappa}^{\varepsilon_\fg,\pm}(\Gamma_0(4N))$ as the $\pm$ \begin{it}genus subspaces\end{it}.
\item For $\kappa=k+\frac{1}{2}$  $(k\in\N_0)$, Kohnen \cite[Proposition 4]{Kohnen-newform} proved that $w_p:=p^{-\frac{k}{2}+\frac{1}{4}}|U_p W_p$ is an involution on  $\sM_{k+\frac12}$ $(\Gamma_0(4N))$ and that the $\pm 1$-eigenspaces are subspaces of $\sM_{k+\frac12}(\Gamma_0(4N))$ consisting of modular forms whose $n$-th Fourier coefficient vanishes if \!\footnote{This corrects a typo in \cite[Proposition 4]{Kohnen-newform}.}
$
(\frac{(-1)^kn}p)=\mp(\frac{\frac{N}{p}}p).
$
  Therefore, the subspaces defined above are common eigenspaces of $w_p$ $(p|N)$. If $f\in \sM_{k+\frac{1}{2}}^{\fg,\varepsilon}(\Gamma_0(4N))$, then
  \begin{equation} \label{eq: eigenvalues}
  f|w_p=\JS{\varepsilon(-1)^k\frac{N}{p}}p\chi_{p^*}(\fg)f.
  \end{equation}
\item For $\kappa=k\in\N_0$, $\sM_k^{\fg,\varepsilon}(\Gamma_0(4N),\chi_{(-1)^k4N})$ are common eigenspaces of $p^{-\frac{k}{2}+\frac{1}{2}}|U_p W_p$. By Lemma \ref{lemma: g eigenvalues} \textup{(2)}, if $f\in\mM_k^{\fg,\pm}(\Gamma_0(4N),\chi_{(-1)^k4N}),$
 then
   \begin{equation}\label{eqn:fGpmUpWp}
  f|p^{-\frac{k}{2}+\frac{1}{2}}U_p W_p=\pm \varepsilon_p^{\pm 1}\chi_{p^*}(\fg)f.
  \end{equation}
In particular, the corresponding eigenvalues belong to $\{ \pm 1, \pm i\}. $
\end{enumerate}
\end{remarks}
\begin{Example}
We decompose the space $\sM_{\frac{3}{2}}(\Gamma_0(12))$ into subspaces by computing $\sM_{\frac{3}{2}}^{\varepsilon,\pm}(\Gamma_0(12))$ for $\varepsilon\in \mathcal{E}_3$. First note that since the Nebentypus character is trivial, the character $\varphi$ in \eqref{eqn:epsilonchi1p} is trivial for all primes $p$. We next determine the set $\mathcal{E}_{3}$. Since $N=3$, we have $\delta_N=3$, and hence by \eqref{eqn:epsilongenus}
\[
\mathcal{E}_3=\left\{\varepsilon:\{3\}\mapsto \{\pm 1\}: \prod_{p\mid 3}\varepsilon(p)=1\} = \{\varepsilon:\{3\}\mapsto \{\pm 1\}: \varepsilon(3)=1\right\}
\]
is trivial.  We have $\varepsilon=\varepsilon_{\fg}\in \mathcal{E}_{3}$ with $\fg$  the principal genus of $\mathcal{Q}_{-3}$ . The principal genus consists of a single class with representative $Q_3(\bm{m})=m_1^2+m_1m_2+m_2^2$. The theta function
\[
\Theta_{Q_3}(\tau)=\sum_{\bm{m}\in\Z^2} q^{Q_3(\bm{m})}=\sum_{n\geq 0} c(n)q^n
\]
satisfies $c(n)=0$ if $\chi_{-3}(n)=-1$. Hence
\[
f_1(\tau):=\Theta_{Q_3}(4\tau)\Theta(3\tau)=1+2q^3+6q^4+12q^7+8q^{12}+12q^{15}+18q^{16}+O\left(q^{17}\right)\in \sM_{\frac{3}{2}}^{\fg,+}(\Gamma_0(12)).
\]
 Using Magma \cite{Magma}, we compute a basis for $\sM_{\frac{3}{2}}(\Gamma_0(12))$ consisting of $f_1$ and
\begin{align*}
      f_2(\tau)&=q+q^3+2q^4+2q^6+2q^7+q^9+4q^{10}+2q^{12}+4q^{13}+O\left(q^{14}\right),\\
      f_3(\tau)&=q^2-q^4+2q^5+q^6-2q^7+q^8+2q^9+2q^{11}-q^{12}+4q^{14}+O\left(q^{15}\right).
\end{align*}
Using the valence formula, we verify that if $\gcd(n,12)=1$ and $\chi_{-3}(-n)=-1$ (i.e., $n\equiv 1\Pmod{6}$), then the $n$-th Fourier coefficient of
  $$
  f_1(\tau)+6f_3(\tau)=1+6q^2+2q^3+12q^5+6q^6+6q^8+12q^9+12q^{11}+O\left(q^{12}\right)
  $$
vanishes. Thus $f_1+6f_3\in \sM_{\frac{3}{2}}^{\fg,-}(\Gamma_0(12))$. Note that the above construction demonstrates how elements of $ \sM_{\frac{3}{2}}^{\fg,+}(\Gamma_0(4N))$ may be naturally constructed via theta functions, whereas a general construction of elements in $ \sM_{\frac{3}{2}}^{\fg,-}(\Gamma_0(4N))$ is not known. 
\end{Example}

\noindent The following compatibility of the genus subspaces with the operators $U_d$ and $W_p$ is used repeatedly. 
\begin{Lemma} \label[Lemma]{lemma: UW on g}
  Let $N\in\N$ be odd and squarefree,  $\fg$ be
  a genus of binary quadratic forms of discriminant $-4N$, $d$ be a positive divisor of $N$, and $p$ be a prime divisor of $N$ with $p\nmid d$. 
\begin{enumerate}
\item 
For $f\in\mM_{k+\frac{1}{2}}^{\fg,\pm}(\Gamma_0(4N))$, we have
\[
    \begin{split}
      p^{\frac{k}{2}-\frac{1}{4}}f|U_d W_p
      =\varepsilon_p^{\pm}\chi_{p^*}(\fg)f|U_{dp}.
    \end{split}
\]
\item 
For $g\in  \mM_k^{\fg,\pm}(\Gamma_0(4N),\chi_{(-1)^k4N})$, we have
\[
    p^{\frac{k}{2}-\frac{1}{2}}g|U_d W_p=\varepsilon_p^{\pm}\JS{d}p\chi_{p^*}(\fg)g|U_{dp}.
\]
\end{enumerate}
\end{Lemma}
\begin{proof}
(1)
 Since $W_p$ is an involution, the claim is equivalent to
  \begin{equation} \label{eq: goal}
    p^{\frac{k}{2}-\frac{1}{4}}f|U_{d} W_{p}^2=\varepsilon_p^{\pm}\chi_{p^*}(\fg)f|U_{dp} W_p.
  \end{equation}
Extending \cite[Proposition 1.5]{Shimura-correspondence} to real-analytic modular forms gives $f|U_{d}\in\mM_{k+\frac{1}{2}}(\Gamma_0(4N),\chi_{4d})$. Thus, by Lemma \ref{lemma: Atkin-Lehner} (1),
  $$
  	p^{\frac{k}{2}-\frac{1}{4}}f|U_{d} W_{p}^2=p^{\frac{k}{2}-\frac{1}{4}}\varepsilon_p^{2k+1}\JS{\frac{N}{p}}p\JS{d}pf|U_{d}.
  $$
  On the other hand, by the extension of \cite[Proposition 1.5]{Shimura-correspondence} again, $f|U_p\in\mM_{k+\frac{1}{2}}(\Gamma_0(4N),\chi_{4p})$. The decomposition of the character $\chi_{4p}$ into a product of two characters of modulus
  $4\frac{N}{p}$ and $p$ is
  $$
  \JS{4p}{n}=\JS{4(-1)^{\frac{p-1}{2}}}{n}\JS{n}{p}.
  $$
  Therefore, by Lemma \ref{lemma: Atkin-Lehner} (2) and \eqref{eq: eigenvalues}, we have
  \begin{equation*}
      f|U_{dp} W_p = f|U_p U_{d} W_p=\JS{d}{p}f|U_p W_p U_{d} = p^{\frac{k}{2}-\frac{1}{4}}\JS dp\JS{\pm(-1)^{k}\frac{N}{p}}p\chi_{p^*}(\fg)f|U_{d}.
  \end{equation*}
This proves \eqref{eq: goal}, which implies the claim.

\noindent
(2) Replacing \Cref{lemma: Atkin-Lehner} and \eqref{eq: eigenvalues} with \Cref{lemma:W^2 2} and \eqref{eqn:fGpmUpWp}, this claim follows similarly.
\end{proof}

We next show that $\mathcal H_{D,N}$ lie in certain genus subspaces. For this, we define $\varepsilon_{D,N}\in \mathcal{E}_{DN}$ by
\[
\varepsilon_{D,N}(p):=
\begin{cases}
\left(\frac{-1}{p}\right)^{\omega(D)}&\text{if }p\mid N,\\
-\left(\frac{-1}{p}\right)^{\omega(D)}&\text{if }p\mid D,\\[.5em]
(-1)^{\omega(D)}&\text{if $p=2$ and $DN\equiv 1\pmod{4}$}.
\end{cases}
\]
\begin{Proposition} \label[Proposition]{proposition: genus of HDN}
  Assume that $D$, $N\in\N$ are odd, squarefree, and coprime.
\noindent

\noindent
\begin{enumerate}
\item For $\omega(D)$ even,
  $\widehat{\mathcal H}_{N}\in\sH_{\frac{3}{2}}^{\varepsilon_{D,N},-}\left(\Gamma_0(4N)\right) \text{ if }D=1 \text{ and }
    \mathcal H_{D,N}\in\sM_{\frac{3}{2}}^{\varepsilon_{D,N},-}\left(\Gamma_0(4DN)\right)
    \text{ if }D>1.$
\item
   If $\omega(D)$ is odd, then $\mathcal H_{D,N}\in \sM_{\frac{3}{2}}^{\varepsilon_{D,N},+}(\Gamma_0(4DN))$.
\end{enumerate}
\end{Proposition}
\begin{proof}
(1) Let $n\in \N$ be coprime to $4DN$. From the definitions \eqref{eqn:HDNcoeff} and \eqref{eq: hDN}, we see that the $n$-th Fourier coefficient $H_{D,N}(n)$ of $\mathcal H_{D,N}$ is only possibly nonzero if $n$ satisfies $n\equiv0, 3\Pmod 4$ and
\begin{equation}\label{eqn:vanish-n}
\JS{-n}p=\begin{cases}
  1 &\text{if }p|N, \\
  -1 &\text{if }p|D. \end{cases}
\end{equation}
Thus $\mathcal H_{D,N}\in\sM_{\frac{3}{2}}^{\varepsilon_{D,N},-}(\Gamma_0(4DN))$ if $D>1$. If $D=1$, then we must also verify that the coefficients of the non-holomorphic part satisfy the claimed vanishing conditions. Note that in this case $\varepsilon_{D,N}(p)=1$ for all $p\mid \delta_{N}$. It follows from the definition that the $n$-th Fourier coefficient of the non-holomorphic part of $\widehat{\mathcal H}_{N}$ vanishes if $-n$ is not a square $\!\!\Pmod{4N}$.  Thus $\widehat{\mathcal H}_{1,N}\in\sH_{\frac{3}{2}}^{\varepsilon_{D,N},-}(\Gamma_0(4N))$.
\noindent

\noindent (2) Suppose that $\omega(D)$ is odd and note that
\[
\left(\frac{-n}{p}\right)=\left(\frac{-1}{p}\right) \left(\frac{n}{p}\right)=\left(\frac{-1}{p}\right)^{\omega(D)} \left(\frac{n}{p}\right).
\]
Rearranging \eqref{eqn:vanish-n}, we conclude that $H_{D,N}(n)=0$ unless $n\equiv 0,3\pmod{4}$ and
\[
\JS{n}p=\begin{cases}
  \left(\frac{-1}{p}\right) &\text{if }p|N, \\[-12pt]\\
  -\left(\frac{-1}{p}\right) &\text{if }p|D.
 \end{cases}
\]
Hence if $H_{D,N}(n)\neq 0$, then $\JnoS{n}{p}=\varepsilon_{D,N}(p)$ for all $p\mid DN$ and $\chi_{-4}(n)=-1=\varepsilon_{D,N}(2)$ if $DN\equiv 1\pmod{4}$. We conclude that $\mathcal{H}_{D,N}\in \sM_{\frac{3}{2}}^{\varepsilon_{D,N},+}(\Gamma_0(4DN))$.
\end{proof}

\section{A level-lowering theorem for genus subspaces}\label{sec:mainthm}
 The main technical difficulty in proving Theorems \ref{theorem: D>1} and \ref{theorem: principal genus} is the following general level-lowering theorem, which is the main goal of this section.
\begin{Theorem} \label{theorem: level reduction}
 Let $N\in\N$ be odd and squarefree, $\fg$ be a genus of binary quadratic forms of discriminant $-4N$, $\varepsilon\in\{\pm1\}$, and $k_1$, $k_2 \in \N_{0}$. For $f\in \mM_{k_1+\frac{1}{2}}^{\fg,\varepsilon}\!(\Gamma_0(4N))$ and $g\in \mM_{k_2}^{\fg,-\varepsilon}(\Gamma_0(4N),\chi_{(-1)^k4N})$,
  \begin{equation*} 
  \sum_{d|N}\left(f|U_d \cdot g|U_d \right)|U_{\frac{N}{d}}
  \in\mM_{k_1+k_2+\frac{1}{2}}\left(\Gamma_0(4)\right).
  \end{equation*}
\end{Theorem}
\begin{remark}
	The key point is that the genus condition forces the local Atkin\textendash Lehner contribution at every prime $p\mid N$ to cancel, allowing the level to drop from $4N$ to $4$.
\end{remark}

We now introduce the operators needed for the level-lowering process. Throughout this section, we fix an odd and squarefree $N\in\N$. In the process of proving Theorem \ref{theorem: level reduction} we encounter spaces of modular forms on $\Gamma_0(4M)$, $M|N$. We adopt the following notation to distinguish the $W_p$-operators on modular forms of different levels.

\begin{Definition*} 
Suppose that $p\mid N$ and $p^2\nmid N$, and let $r$, $s\in\Z$  satisfy $sp-\frac{4rN}{p}=1$. For a positive divisor $M|N$ with $p|M$, we let 
  $$
 	 W_{k,p,M}:=\left(\M p{\frac{rN}{M}}{4M}{sp},\phi_{k,p,M}\right)\in \mathfrak{G}_{k},
  $$
where 
$\phi_{k,p,M}(\tau):=\varepsilon_p^{-2k}p^{-\frac{k}{2}}(4M\tau+sp)^k.$ As usual, we omit the dependence on $k$ if it is clear from context. If $M=N$, then we write $W_p$ as above.
\end{Definition*}
The following lemma compares the level $M$ and level $N$ versions of the $W_p$-operators.
\begin{Lemma} \label[Lemma]{lemma: two W}
  Let $p$ be a prime divisor of $N$ with $p^2\nmid N$ and $M|N$ such that $p|M$. Let $\chi$ be a quadratic Dirichlet character $\hspace{-0.15cm}\Pmod{4M}$ and let $\chi=\chi_1\chi_2$ be the corresponding decomposition in Dirichlet characters $\chi_1$ and $\chi_2$ of modulus $\frac{4M}{p}$ and $p$, respectively. Then,  for $f\in\mM_{k+\frac{1}{2}}(\Gamma_0(4M),\chi)$, 
  $$
  f|W_{p,M}=\chi_2\left(\frac{N}{M}\right)f|W_{p}.
  $$
\end{Lemma}

We next use $W_{p,M}$ to construct a level-lowering operator. Let $M\mid N$ and $p\mid M$ be a prime. Suppose that $\chi$ and $\psi$ are characters $\Pmod{4M}$ and $\Pmod{\frac{4M}{p}}$, respectively, such that for every $\gcd(d,4M)=1$, we have
\begin{equation}\label{eqn:chipsirel}
\chi(d)=
\begin{cases}
\left(\frac{p}{d}\right)\psi(d)&\text{if }\kappa\in \Z+\frac{1}{2},\\
\psi(d)&\text{if }\kappa\in \Z.
\end{cases}
\end{equation}
Then we let $\SS_{\kappa,p,M,\chi}$ denote the operator
  $$
  f|\SS_{\kappa,p,M,\chi}:=
\begin{cases}
f|U_p+p^{\frac{\kappa}{2}-1}\psi(p)\JS{\frac{4M}{p}}{p}
f|W_{p,M}&\text{if }\kappa\in\Z+\frac{1}{2},\\[1em]
f|U_p+\psi(p)\varepsilon_p^{2\kappa} p^{\frac{\kappa}{2}-1}f|W_{p,M}&\text{if }\kappa\in\Z.
\end{cases}
  $$
 The operator is chosen so that the $W_p$-term cancels the unwanted local eigenspace. For brevity, we omit the dependence on $\kappa$ if it is clear from the context. One obtains the following modular properties of $f|\SS_{\kappa,p,M,\chi}$ by an argument that closely resembles the proof of \cite[Lemma 7]{Atkin-Lehner}.
\begin{Lemma} \label[Lemma]{lemma: level reduction}  
Let $M| N$ such that $p|M$, $\kappa\in\frac{1}{2}\Z$, and $\chi\Pmod{4M}$ and $\psi\Pmod{\frac{4M}{p}}$ be characters satisfying \eqref{eqn:chipsirel}. Then,  for $f\in \mM_{\kappa}(\Gamma_{0}(4M),\chi)$,  we have
\[
f|\SS_{\kappa,p,M,\chi}\in \mM_{\kappa}\left(\Gamma_0\left(\frac{4M}{p}\right),\psi\right).
\]
\end{Lemma}

We next use Lemma \ref{lemma: level reduction} with $\kappa=k+\frac{1}{2}\in\Z+\frac{1}{2}$, $\chi(n)=(\frac{4M}{n})$, and $\psi(n)=(\frac{\frac{4M}{p}}{n})$. We hence first verify that, for $\gcd(d,4M)=1$,
\[
\left(\frac{p}{d}\right)\chi(d)=\left(\frac{p}{d}\right)\left(\frac{4M}{d}\right)=\left(\frac{\frac{4M}{p}}{d}\right)=\psi(d).
\]
Since $\psi(d)$ only depends on $d\Pmod{\frac{4M}{p}}$, we see that $\chi$ and $\psi$ satisfy \eqref{eqn:chipsirel} for $\kappa\in\Z+\frac{1}{2}$ and we may hence apply Lemma \ref{lemma: level reduction}. By induction, Lemma \ref{lemma: level reduction} then shows that if $N=p_1\ldots p_\ell$ is the prime decomposition of $N$ and $f\!\in \mM_{k+\frac{1}{2}}\!\left(\Gamma_0(4N),\chi_{4N}\right)$, then, abbreviating $\SS_{p,M}=\SS_{k+\frac{1}{2},p,M,\chi_{4M}}$,
\begin{equation}\label{eqn:LowerRepeat}
f|\SS_{p_1,N} \SS_{p_2,\frac{N}{p_1}} \SS_{p_3,\frac{N}{p_1p_2}}\ldots
\SS_{p_\ell,p_\ell}\in\mM_{k+\frac{1}{2}}(\Gamma_{0}(4)).
\end{equation}
\begin{proof}[Proof of Theorem \ref{theorem: level reduction}]

 Let $N=p_1\ldots p_\ell$ be the prime factorization of $N$. If $f\in \mM_{k_1+\frac{1}{2}}(\Gamma_0(4N))$ and
  $g\in \mM_{k_2}(\Gamma_0(4N),\chi_{(-1)^{k_2}4N})$, then
  $fg\in \mM_{k_1+k_2+\frac{1}{2}}(\Gamma_{0}(4N),\chi_{4N})$. Thus, by repeated use of Lemma \ref{lemma: level reduction} (see \eqref{eqn:LowerRepeat}),
  \begin{equation} \label{eq: equivalent 2}
  (fg)|\SS_{p_1,N} \SS_{p_2,\frac{N}{p_1}} \SS_{p_3,\frac{N}{p_1p_2}}\ldots \SS_{p_n,p_n}\in \mM_{k_1+k_2+\frac{1}{2}}(\Gamma_{0}(4)).
  \end{equation}
  We prove by induction on $m$ that
  \begin{equation} \label{eq: induction}
  (fg)|\SS_{p_1,N}\ldots \SS_{p_m,\frac{N}{p_1}\ldots p_{m-1}}
  =\sum_{d|p_1\ldots p_m}
  \left((f|U_d)\cdot(g|U_d)\right)|U_{\frac{p_1\ldots p_m}{d}}.
  \end{equation}
Combining \eqref{eq: induction} with \eqref{eq: equivalent 2} then yields the claim. For  $m=1$,  \eqref{eq: induction} becomes
  $$
  (fg)|\SS_{p_1,N}=(fg)|U_{p_1}
  +p_1^{\frac{k_1}{2}-\frac{1}{4}}f|W_{k_1+\frac{1}{2},p_1,N}\cdot
  p_1^{\frac{k_2}{2}-\frac{1}{2}}g|W_{k_2,p_1,N}.
  $$
  By Lemma \ref{lemma: UW on g}, this equals
  $$
  (fg)|U_{p_1}+f|U_{p_1}\cdot g|U_{p_1},
  $$
  which is exactly what \eqref{eq: induction} claims for $m=1$. Now assume that \eqref{eq: induction} holds up to some $m<\ell$ and set  $M=p_1\ldots p_m$. Since $(fg)|\SS_{p_1,N}\ldots
  \SS_{p_m,\frac{N}{p_1\ldots p_m}}\in
  \mM_{k_1+k_2+\frac{1}{2}}(\Gamma_0(\frac{4N}{M}),\chi_{\frac{4N}{M}})$ by Lemma \ref{lemma: level reduction}, the induction hypothesis and Lemma \ref{lemma: two W} imply that
  \begin{align}
      &\hspace{-.5cm}(fg)|\SS_{p_1,N}\ldots \SS_{p_{m+1},\frac{N}{p_1\ldots p_{m+1}}}\label{eq: fS}\\
      &=\sum_{d|M}\left(f|U_d\cdot g|U_d\right)|U_{\frac{M}{d}} U_{p_{m+1}}+p_{m+1}^{\frac{k_1}{2}+\frac{k_2}{2}-\frac{3}{4}}\JS M{p_{m+1}}
        \sum_{d|M}\left(f|U_d g|U_d\right)|
      U_{\frac{M}{d}} W_{p_{m+1},N}.\nonumber
  \end{align}
Applying Lemma \ref{lemma: Atkin-Lehner} (2) and then Lemma \ref{lemma: UW on g}, we may rewrite the second sum in \eqref{eq: fS} as
\begin{equation*}
	p_{m+1}^{\frac{k_1}{2}+\frac{k_2}{2}-\frac{3}{4}}\sum_{d|M}\JS d{p_{m+1}}
	\!\! \left(f|U_d W_{k_1+\frac{1}{2},p_{m+1},N}\cdot
	g|U_d W_{k_2,p_{m+1},N}\right)\!\!\Big|U_{\frac{M}{d}}\!=\!\!
	\sum_{d|M}\left(f|U_{dp_{m+1}}\cdot
	g|U_{dp_{m+1}}\right)\!\!\Big|U_{\frac{M}{d}}.
\end{equation*}
Plugging this into \eqref{eq: fS}, we conclude that \eqref{eq: induction} also holds for $m+1$. This proves \eqref{eq: induction} by induction and the proof of the theorem is complete.
\end{proof}

\section{Proofs of Theorem \ref{thm:D>1} and Theorem \ref{theorem: D>1}}\label{sec:MainD>1}
\begin{proof}[Proof of Theorem \ref{theorem: D>1}]
By \cite[Proposition 2.1]{Shimura-correspondence}, for $Q\in\mathcal{Q}_{-\delta}$, we have $\Theta_Q\in \sM_1(\Gamma_0(\delta), \chi_{-\delta})$, so
  $\Theta_Q|V_{4}\in\sM_1(\Gamma_0(4\delta),\chi_{-4\delta})$. For brevity, write $\fg=\fg_{-\delta,D,N}$ (see \eqref{eqn:fgdef}). For $Q\in\fg$, we have $\Theta_Q|V_{4}\in\sM_1^{\fg,+}(\Gamma_0(4\delta),\chi_{-4\delta})$. On the other hand, by Proposition \ref{proposition: genus of HDN}, $\mathcal H_{D,N}(\tau)=\sum_{n\geq0} H_{D,N}(n)q^n\in\sM_{\frac{3}{2}}^{\fg,-}(\Gamma_0(4\delta))$. Thus, by Theorem \ref{theorem: level reduction}, we have
  \begin{equation} \label{eq: in M5/2}
  \sum_{d|\delta}\big(\mathcal H_{D,N}|U_d
  \cdot\Theta_Q|V_{4} U\big)|U_{\frac{\delta}{d}}
  \in\sM_{\frac{5}{2}}\left(\Gamma_0(4)\right).
  \end{equation}
 Since $\mathcal H_{D,N}$ is a modular form whose $n$-th Fourier coefficient vanishes if $n\equiv 1,2\Pmod 4$, while $\delta \equiv 3\Pmod4$, the $n$-th Fourier coefficient of  \eqref{eq: in M5/2} can only be nonzero if $n\equiv 0,1\Pmod 4$. Thus \eqref{eq: in M5/2} belongs to Kohnen's $+$-space of weight $\frac{5}{2}$ on $\Gamma_0(4)$ and hence must be a multiple of $\mathcal H_{\frac{5}{2}}$.
 Comparing the constant terms, we conclude that
  $$
  \sum_{d|\delta}\big(\mathcal H_{D,N}|U_d
  \cdot\Theta_Q|V_{4} U_d\big)|U_{\frac{\delta}{d}}
  =120\cdot 2^{\omega(\delta)}H_{D,N}(0)\mathcal H_{\frac{5}{2}}.
  $$
  This completes the proof of Theorem \ref{theorem: D>1}.
\end{proof}
We next write the modular identity in Theorem \ref{theorem: D>1} in terms of classical Hurwitz class number generating functions. 
\begin{prop}\label[Proposition]{prop:Hrewrite}
Let $\delta\in\N$ be squarefree with $\delta\equiv 3\Pmod 4$ and suppose that $(D,N)\in\mathscr{S}_{\delta}$.  Then for $Q\in\fg_{-\delta,D,N}$, $M\in\N$, and $m\in\Z$, we have
  \begin{equation*} 
\sum_{d|\delta}\big( \mathcal H_{D,N}|U_d \cdot \Theta_{Q,M,m}|  V_{4} U_d\big)|U_{\frac{\delta}{d}} =  2^{\omega(\delta)} \sum_{d|\delta}(-1)^{\omega(\gcd(D,d))}d\cdot \left(\mathcal H|V_{d^2}\cdot \Theta_{Q,M,m}|V_4\right)|U_{\delta}.
  \end{equation*}
\end{prop}
\begin{proof}
Using the identity $f|U_d=f|U_d V_d U_d$, one can rewrite the left-hand side of the claim as
  \begin{equation*}
      \sum_{d|\delta}\big(\mathcal H_{D,N}|U_d V_d \cdot\Theta_{Q,M,m}| V_{4}\big)|U_{\delta}.
  \end{equation*}
  We isolate the class-number factor by setting $F:=\sum_{d|\delta}\mathcal H_{D,N}|U_d V_d$.   By \Cref{lemma: alternative HDN}, we have 
    \begin{align}
\nonumber       F(\tau)&=(-1)^{\omega(D)}\sum_{d|\delta}\sum_{r|\delta}r\sum_{\substack{n\geq0\\d|r^2n}}b_{D,N,r}(n)H(n)q^{r^2n}
      =(-1)^{\omega(D)}\sum_{r|\delta}r\sum_{n\geq0}
      b_{D,N,r}(n)H(n)\!\!\!\sum_{d|\gcd(\delta,r^2n)}\!\!\!q^{r^2n} \\
      &=(-1)^{\omega(D)}\sum_{r|\delta}r\sum_{n\geq0} 2^{\omega(\gcd(\delta,rn))}b_{D,N,r}(n)
      H(n)q^{r^2n}.\label{eqn:Fsum}
    \end{align}
  Note that if $n\in\N_{0}$ such that $\JnoS{-n}p=1$ for some $p|D$ or $\JnoS{-n}p=-1$ for some $p|N$, then for any $p\nmid r$ we have $b_{D,N,r}(n)=0$. For  $n\in\N_0$ for which there does not exist $p$ with $\JnoS{-n}p=1$ for some $p|\frac{D}{\gcd(D,r)}$ or $\JnoS{-n}p=-1$ for some $p|\frac{N}{\gcd(N,r)}$, we have
  \begin{equation*}
      b_{D,N,r}(n)
      =(-1)^{\omega\left(\gcd\left(D,\frac{\delta}{r}\right)\right)}2^{\omega\left(\frac{\delta}{r}\right)-\omega\left(\gcd\left(\frac{\delta}{r},n\right)\right)} =(-1)^{\omega(D)+\omega\left(\gcd(D,r)\right)}2^{\omega(\delta)-\omega\left(\gcd(\delta,rn)\right)}.
  \end{equation*}
  Therefore, plugging this into \eqref{eqn:Fsum} yields
  $$
  F(\tau)=2^{\omega(\delta)}\sum_{r|\delta}(-1)^{\omega(\gcd(D,r))}r\sum_{n\geq 0}  \deltabeta_{D,N,r}(n)H(n)q^{r^2n},
  $$
  where 
  $$
  \deltabeta_{D,N,r}(n)=\begin{cases}
    1 &\text{if }\JS{-n}p\neq 1\text{ for all }p\left|\frac{D}{\gcd(D,r)}\text{ and }
    \JS{-n}p\neq-1\text{ for all }p\right|\frac{N}{\gcd(N,r)}, \\
    0 &\text{otherwise}. \end{cases}
  $$
Let
  $$
  G(\tau)=\sum_{r|\delta}(-1)^{\omega(\gcd(D,r))}r\mathcal H\left(r^2\tau\right) =\sum_{r|\delta}(-1)^{\omega(\gcd(D,r))}r\sum_{n\geq0}H(n)q^{r^2n}.
  $$
  Note that if $n\in\N$ satisfies $\JnoS{-n}p=1$ for some $p|D$ or $\JnoS{-n}p=-1$ for some $p|N$, then Lemma \ref{lem:genuscharnotrelprime} shows that $n+4Q(\bm{n})$ is never a multiple of $\delta$ for any $n_{1},n_{2}\in\Z$. This shows that
  $$
  \left(\left(F-2^{\omega(\delta)} G\right)\Theta_{Q,M,m}|V_{4}\right)
  |U_{\delta}=0.
  $$
Rearranging, we conclude that 
\[
\left(F\cdot \Theta_{Q,M,m}|V_4\right)|U_{\delta}= 2^{\omega(\delta)}\left(G\cdot \Theta_{Q,M,m}|V_4\right)|U_{\delta},
\]
which gives the claim.
\end{proof}
The following corollary immediately implies Theorem \ref{thm:D>1} by comparing Fourier coefficients.
\begin{Corollary} \label{corollary: D>1}
Let $\delta\in\N$ be squarefree with $\delta\equiv 3\Pmod 4$ and suppose that $(D,N)\in\mathscr{S}_{\delta}^{\operatorname{np}}$.  Then, for $Q\in\fg_{-\delta,D,N}$,  we have
  \begin{equation*} 
  \sum_{d|\delta}(-1)^{\omega(\gcd(D,d))}d\cdot \left(\mathcal H|V_{d^2}\cdot
    \Theta_Q|V_4\right)|U_{\delta}
  =120 H_{D,N}(0)\mathcal H_{\frac{5}{2}}.
  \end{equation*}
\end{Corollary}
\begin{proof}
By Theorem \ref{theorem: D>1}, for $D>1$,  we have 
\[
\sum_{d|\delta}\big(\mathcal H_{D,N}|U_d \cdot\Theta_Q|V_{4} U_d\big)|U_{\frac{\delta}{d}}=2^{\omega(\delta)+3}\cdot15H_{D,N}(0)\mathcal H_{\frac{5}{2}}.
\]
Plugging in Proposition \ref{prop:Hrewrite} with $M=1$, the left-hand side may be rewritten as 
\[
  2^{\omega(\delta)}\sum_{d|\delta}(-1)^{\omega(\gcd(D,d))}d\cdot \left(\mathcal H|V_{d^2}\cdot\Theta_Q|V_4\right)|U_{\delta}.
\]
Cancelling $2^{\omega(\delta)}$ from both sides then yields the claim. 
\end{proof}

\section{Proofs of Theorem \ref{thm:PrincipalGenus} and Theorem \ref{theorem: principal genus}}\label{sec:MainPrincipal}
We now turn to the remaining case $D=1$, which requires a different argument from that used in Section 6. Let $N\in\N$ be odd and squarefree.
 We first establish Theorem \ref{theorem: principal genus}. Its proof requires the following lemma on the sieving operator $S_{8,j}$. 
\begin{Lemma} \label[Lemma]{lemma: level 16}
  Suppose that $f\in\mM_{k+\frac{1}{2}}(\Gamma_{0}(16))$. If $k$ is even and $j\in\{1,5\}$ or if $k$ is odd and $j\in\{3,7\}$, then $f|S_{8,j}\in \mM_{k+\frac{1}{2}}(\Gamma_{0}(16))$.
\end{Lemma}

\begin{proof}
  One checks that $f|S_{8,j}|\gamma^\ast=f|S_{8,j}$ for any $\gamma\in\Gamma_0(64)$, so it suffices to show $f|S_{8,j}|\SM10{16}1^\ast=f|S_{8,j}$ under the assumption on $k$ and $j$. Let $\zeta_8:=e^{\frac{2\pi i}{8}}$. We have 
  $$
  f|S_{8,j}|\M10{16}1^\ast (\tau)
  =\frac18(16\tau+1)^{-k-\frac{1}{2}}\sum_{\ell=0}^7\zeta_8^{-j\ell}
  f\left(\M1{\frac{\ell}{8}}01\M10{16}1\tau\right).
  $$
  Let $\alpha_\ell$, $\beta_\ell \in \mathbb{Z}$ ($0\leq \ell\leq 7$) with $\alpha_\ell(1+2\ell)+16\beta_\ell=1$. Set $\lambda:=\alpha_\ell\ell+8\beta_\ell$ and $\gamma_\ell:=\SM{1+2\ell}{-\beta_{\ell}}{16}{\alpha_{\ell}}$ so
  $$
  \M1{\frac{\ell}{8}}01\M10{16}1=\gamma_\ell
  \M1{\frac{\lambda}{8}}01.
  $$
  Then 
  $$
  f|S_{8,j}|\M10{16}1^\ast(\tau)
  =\frac18\sum_{\ell=0}^7\zeta_8^{-j\ell}
  \varepsilon_{1+2\ell}^{-2k-1}
  f|\gamma_\ell^\ast\left(\tau+\frac{\lambda}{8}\right).
  $$
  Now $\varepsilon_{1+2\ell}^{-2k-1}=1$ if $\ell$ is even and it equals $i^{-2k-1}$ if $\ell$ is odd. Also, we check that $\lambda\equiv \ell\Pmod 8 $ if $\ell$ is even and $\lambda =\ell + 2$ if $\ell$ is odd. Thus, we have
  $$
  f|S_{8,j}|\M10{16}1^\ast(\tau)
  =\frac18\sum_{\ell\in\{0,2,4,6\}}\zeta_8^{-j\ell}f\left(\tau+\frac{\ell}{8}\right)
  +\frac18\sum_{\lambda\in\{1,3,5,7\}}\zeta_8^{-j(\lambda-2)}
  i^{-2k-1}f\left(\tau+\frac{\lambda}{8}\right).
  $$
  From this, we see that $f|S_{8,j}|\SM10{16}1^\ast =f|S_{8,j}$ under our assumption about $k$ and $j$.
\end{proof}

We are now ready to prove a weaker version of \Cref{theorem: principal genus} with unspecified constants.
\begin{Theorem} \label{thm:PrincipalGenus-weak}
  Let $N\in\N$ be squarefree with $N\equiv 3\Pmod 4$. For $Q\in \fg_{-N,1,N}$, there exist $c_{Q,1},c_{Q,5}\in\C$ such that 
  $$
   G_{Q,1}=c_{Q,1}2^{\omega(N)}\sigma_1(N)\mathcal H_{\frac{5}{2}}|S_{8,1}, \qquad
  G_{Q,5}=c_{Q,5} 2^{\omega(N)}\sigma_1(N)\mathcal H_{\frac{5}{2}}|S_{8,5}.
  $$
\end{Theorem}
\begin{proof}
  We only prove the claim for $G_{Q,1}$ with $N\equiv 3\Pmod 8$; the remaining cases are shown similarly. By Proposition \ref{proposition: genus of HDN}, $\widehat{\mathcal H}_{N}\in\sH_{\frac{3}{2}}^{\varepsilon_{1,N},-}(4N)$. Since $\varepsilon_{1,N}(p)=1$ for $p\mid \delta_{N}$, $\widehat{\mathcal H}_{N}\in\sH_{\frac{3}{2}}^{\fg,-}(4N)$, where $\fg=\fg_{-\delta_N,1,N}$ is the principal genus of binary quadratic forms discriminant $-\delta_N$. Note that $\Theta_{Q,2,0}|V_4\in\sM_1(\Gamma_0(4N),\chi_{-4N})$ and Lemma \ref{lem:genuscharnotrelprime} implies that if $\JnoS np=-1$ for some prime $p\mid N$, the $n$-th coefficient of $\Theta_{Q,2,0}|V_4$ vanishes.
 Thus $\Theta_{Q,2,0}|V_4\in \sM_{1}^{\fg,+}(\Gamma_0(4N),\chi_{-4N})$. Hence, 
  $$
  F=\sum_{d|N}\Big(\widehat{\mathcal H}_N|U_d\cdot\Theta_{Q,2,0}|V_4 U_d\Big)|U_{\frac{N}{d}}\in \mM_{\frac{5}{2}}\left(\Gamma_0(16)\right),
  $$
  by Theorem \ref{theorem: level reduction}. Then, by Lemma \ref{lemma: level 16},
  $$
  \widehat G_{Q,1}:=F|S_{8,1}=\sum_{d|N}
  \Big(\widehat{\mathcal H}_N|U_d
      \cdot\Theta_{Q,2,0}|V_4 U_d\Big)
      |U_{\frac{N}{d}}|S_{8,1}\in \mM_{\frac{5}{2}}\left(\Gamma_0(16)\right).
  $$
Note that, by \eqref{eqn:HnHat}, the $n$-th Fourier coefficient of the non-holomorphic part of $\widehat{\mathcal H}_{N}$ vanishes if $-n$ is not a square.  Since the $m$-th Fourier coefficient of $\Theta_{Q,2,0}|V_4$ vanishes unless $8\mid m$ and $N\equiv 3\pmod{8}$ by assumption, we conclude that the non-holomorphic part of $F$ is annihilated by $S_{8,1}$. Hence $\widehat G_{Q,1}$ is holomorphic. We conclude that $\widehat{G}_{Q,1}\in \sM_{\frac{5}{2}}(\Gamma_0(16))$. 

  Finally, by Lemma \ref{lemma: level 16}, $\mathcal{H}_{\frac{5}{2}}|S_{8,1}\in\sM_{\frac{5}{2}}(\Gamma_0(16))$. Using Magma \cite{Magma}, we find that the space $\sM_{\frac{5}{2}}(\Gamma_0(16))$ is $6$-dimensional and the subspace of forms whose Fourier expansions are supported on powers of $q$ congruent to $1\pmod{8}$ is one-dimensional. Therefore, $G_{Q,1}$ must be a scalar multiple of $\mathcal{H}_{\frac{5}{2}}|S_{8,1}$. We write the constant of proportionality in the form $c_{Q,1} 2^{\omega(N)}\sigma_1(N)$.\qedhere
\end{proof}

\begin{proof}[Proof of Theorem \ref{thm:PrincipalGenus}]
Suppose that $m\in\{0,1\}$, $j\in\{1,5\}$ and 
\[
\delta\equiv \begin{cases} 3&\text{if }(m,j)\in \{(0,1),(1,5)\}, \\
7&\text{if }(m,j)\in\{(0,5),(1,1)\},
\end{cases} \ \pmod{8}.
\]
Using Proposition \ref{prop:Hrewrite} with $M=2$, we write (note that since $D=1$, we have $\omega(\gcd(D,d))=0$)
\[
G_{Q,j}=  \sum_{d|\delta}\big( \mathcal H_{\delta}|U_d \cdot \Theta_{Q,m,2}| V_{4} U_d\big)|U_{\frac{\delta}{d}} S_{8,j} =  2^{\omega(\delta)} \sum_{d|\delta}d\cdot \left(\mathcal H|V_{d^2}\cdot \Theta_{Q,m,2}|V_4\right)|U_{\delta} S_{8,j}.
\]
By Theorem \ref{thm:PrincipalGenus-weak}, there exists a constant $c_{Q,j}$ such that
\[
G_{Q,j}=c_{Q,j}2^{\omega(\delta)}\sigma_1(\delta)\mathcal{H}_{\frac{5}{2}}|S_{8,j}.
\]
Therefore
\begin{equation}\label{eqn:cQjPrincipalGenus}
2^{\omega(\delta)} \sum_{d|\delta}d\cdot \left(\mathcal H|V_{d^2} \cdot\Theta_{Q,m,2}|V_4\right)|U_{\delta} S_{8,j}=c_{Q,j}2^{\omega(\delta)}\sigma_1(\delta)\mathcal{H}_{\frac{5}{2}}|S_{8,j}.
\end{equation}
Comparing Fourier coefficients on both sides gives the claim if $c_{Q,j}=-[j]_8$, which we next show. First suppose that $j=1$ and
\begin{equation}\label{eqn:PrincipalForm}
\mathcal{Q}(\bm{n}):=n_1^2+n_1n_2+\frac{1+\delta}{4} n_2^2=\left(n_1+\frac{n_2}{2}\right)^2+\frac{\delta}{4}n_2^2
\end{equation}
 is the principal form. We compute the constant of proportionality $c_{\mathcal{Q},1}$ by comparing $n=1$ on both sides of \eqref{eqn:cQjPrincipalGenus} with $-[1]_8$ replaced by $c_{\mathcal{Q},1}$. If $n_2\neq 0$, then $\mathcal{Q}(\bm{n})\geq \frac{\delta}{4}$. But then
\[
H\left(\frac{\delta-4\mathcal{Q}(\bm{n})}{r^2}\right)=0
\]
unless $\mathcal{Q}(\bm{n})\leq \frac{\delta}{4}$. Since $\delta$ is odd, we have $\mathcal{Q}(\bm{n})\neq \frac{\delta}{4}$, and we conclude that $n_2=0$. But then $\mathcal{Q}(n_1,0)=n_1^2$. The sum on the left-hand side of Theorem \ref{thm:PrincipalGenus} therefore becomes 
\[
\sum_{\substack{m\in\Z\\ 4m^2\equiv\delta-3\Pmod8}} \sum_{r\mid \delta}r H\left(\frac{\delta-4m^2}{r^2}\right).
\]
Note that if $r\ge2$, then $H(\frac{\delta-4m^2}{r^2})=0$ because $r\mid \delta$ and $\delta$ is odd and squarefree. We conclude that 
\begin{align*}
\sum_{\substack{m\in\Z\\ 4m^2\equiv\delta-3\Pmod8}} \sum_{r\mid \delta}rH\left(\frac{\delta-4m^2}{r^2}\right)=\sum_{\substack{m\in\Z\\ 4m^2\equiv\delta-3\Pmod8}}H\left(\delta-4m^2\right)
=\sum_{\substack{m\in\Z\\ m\equiv \frac{\delta-3}{4}\pmod{2}}}H\left(\delta-4m^2\right).
\end{align*}
We claim that, for $\delta\equiv 3\pmod{4}$, 
\begin{equation}\label{eqn:deltaeval}
\sum_{\substack{m\in\Z \\ m\equiv \frac{\delta-3}{4}\pmod 2}} H\left(\delta-4m^2\right)=\frac{1}{12}\sigma_1(\delta).
\end{equation}
To obtain \eqref{eqn:deltaeval}, we write
\begin{equation}\label{eqn:genfunHs2}
	\sum_{\delta\equiv 3\pmod 4} \sum_{\substack{m\in\Z \\ m\equiv \frac{\delta-3}{4}\pmod 2}} H\left(\delta-4m^2\right)q^{\delta}=\left(\mathcal{H}|S_{8,3}\cdot \vartheta|V_4\right)(\tau).
\end{equation}

Since the non-holomorphic part of $\widehat{\mathcal{H}}$ is supported on negative squares, we see that \eqref{eqn:genfunHs2} is in $M_2(\Gamma_0(64))$. The function
\[
\frac{1}{12}\sum_{\delta\equiv 3\pmod 4} \sigma_1(\delta)q^\delta=\frac{1}{288} E_2|S_{8,3}(\tau)
\]
is also in $M_2(\Gamma_0(64))$. So the claimed identity follows from the valence formula by comparing the first 
$\frac{2}{12}\left[\SL_2(\Z):\Gamma_0(64)\right]=16$
Fourier coefficients. This verifies \eqref{eqn:deltaeval} and gives $c_{\mathcal{Q},1}=-1$.

We next compute $c_{\mathcal{Q},5}$. For this, we compare the case $n=5$ on both sides of \eqref{eqn:cQjPrincipalGenus} with $-[5]_8$ replaced by $c_{\mathcal{Q},5}$. If $|n_2|\geq 3$, then \eqref{eqn:PrincipalForm} implies that $4\mathcal{Q}(\bm{n})\geq 9\delta$, so that 
\[
H\left(\frac{5\delta-4\mathcal{Q}(\bm{n})}{r^2}\right)=0.
\]
Hence, plugging in \eqref{eqn:PrincipalForm} with $-2\leq n_2\leq 2$, the left-hand side of Theorem \ref{thm:PrincipalGenus} with $n=5$ becomes 
\begin{equation}\label{eqn:PrincipalFormSimplify}
\sum_{n_2=-2}^2 \sum_{\substack{n_1\in\Z\\ \left(5-n_2^2\right)\delta-\left(2n_1+n_2\right)^2\equiv 3\pmod{8}}} \sum_{r\mid \delta} r  H\left(\frac{\left(5-n_2^2\right)\delta-\left(2n_1+n_2\right)^2}{r^2}\right).
\end{equation} 
We make the change of variables $n_1\mapsto 2n_1+n_2$ to rewrite the sum over $n_1$ as $n_1\equiv n_2\pmod{2}$ and the condition $(5-n_2^2)\delta-(2n_1+n_2)^2\equiv 3\pmod{8}$ becomes $(5-n_2^2)\delta-n_1^2\equiv 3\pmod{8}$. So, splitting into the cases with $n_2=0$, $|n_2|=1$, and $|n_2|=2$, \eqref{eqn:PrincipalFormSimplify} becomes
\begin{multline*}
	\rule[-22pt]{0.6pt}{38pt} \sum_{\substack{m\in\Z\\ 5\delta-4m^2\equiv 3\pmod{8}}} \sum_{r\mid \delta} r H\left(\frac{5\delta-4m^2}{r^2}\right) +2\sum_{\substack{m\in\Z\\ m\equiv 1\pmod{2}\\ 4\delta-m^2\equiv 3\pmod{8}}} \sum_{r\mid \delta} r H\left(\frac{4\delta-m^2}{r^2}\right) \\[-15pt]
	+2\sum_{\substack{m\in\Z\\ \delta-4m^2\equiv 3\pmod{8}}} \sum_{r\mid \delta} r H\left(\frac{\delta-4m^2}{r^2}\right). \ \rule[-22pt]{0.6pt}{38pt}
\end{multline*}
If $\delta\equiv 3\pmod{8}$ (resp. $\delta\equiv 7\pmod{8}$), then one checks that $m\equiv 1\pmod{2}$ (resp. $m\equiv 0\pmod{2}$) is equivalent to $5\delta-4m^2\equiv 3\pmod{8}$ and $m\equiv 0\pmod{2}$ (resp. $m\equiv 1\pmod{2}$) is equivalent to $\delta-4m^2\equiv 3\pmod{8}$. So for $\delta\equiv 3+4\ell\pmod{8}$ we may simplify \eqref{eqn:PrincipalFormSimplify} as
\begin{multline}\label{eqn:PrincipalFormSimplify3}
\sum_{\substack{m\in\Z\\ m\equiv \ell+1\pmod{2}\\5\delta-4m^2\equiv 3\pmod{8} }} \sum_{r\mid \delta} r H\left(\frac{5\delta-4m^2}{r^2}\right)+2\sum_{\substack{m\in\Z\\ m\equiv 1\pmod{2}\\ 4\delta-m^2\equiv 3\pmod{8}}} \sum_{r\mid \delta} r H\left(\frac{4\delta-m^2}{r^2}\right)\\
+2\sum_{\substack{m\in\Z\\ m\equiv \ell\pmod{2}\\ \delta-4m^2\equiv 3\pmod{8}}} \sum_{r\mid \delta} r H\left(\frac{\delta-4m^2}{r^2}\right).
\end{multline}
In the second and third sums of \eqref{eqn:PrincipalFormSimplify3}, the terms with $r\neq 1$ vanish, while the terms with $r\nmid \gcd(\delta,5)$ in the first sum vanish.  Thus \eqref{eqn:PrincipalFormSimplify3} simplifies as
\begin{multline}\label{eqn:lastsumtoeval}
\sum_{\substack{m\in\Z\\ m\equiv \ell+1\pmod{2}\\ 5\delta-4m^2\equiv 3\pmod{8}}}H\left(5\delta-4m^2\right) + 5\delta_{5\mid \delta}\hspace{-.3cm} \sum_{\substack{m\in\Z\\ m\equiv \ell+1\pmod{2}\\  5\delta-4m^2\equiv 3\pmod{8}}}H\left(\frac{\delta}{5}-4m^2\right)+2\sum_{\substack{m\in\Z\\ m\equiv 1\pmod{2}\\ 4\delta-m^2\equiv 3\pmod{8}}} H\left(4\delta-m^2\right)\\
+2\sum_{\substack{m\in\Z\\ m\equiv \ell\pmod{2}\\ \delta-4m^2\equiv 3\pmod{8}}}  H\left(\delta-4m^2\right).
\end{multline}

We use Lemma \ref{lem:UnarySum} to evaluate \eqref{eqn:lastsumtoeval}. Plugging in Lemma \ref{lem:UnarySum} (2) for the first, second, and last sums and \ref{lem:UnarySum} (1) for the third sum, we obtain that \eqref{eqn:lastsumtoeval} equals
\begin{multline*}
\delta_{5\delta\equiv 3+4(1+\ell)\pmod{8}}\frac{1}{12}\sigma_1(5\delta)+5\delta_{5\mid \delta} \delta_{\frac{\delta}{5}\equiv  3+4(1+\ell)\pmod{8}} \frac{1}{12}\sigma_1\left(\frac{\delta}{5}\right)
+ 2 \delta_{4\delta\equiv 4\pmod{8}} \frac{2}{21}\sigma_1(4\delta) \\ 
+ 2\delta_{\delta\equiv 3+4\ell\pmod{8}}\frac{1}{12}\sigma_1(\delta)=
\frac{1}{12}\sigma_1(5\delta)+5\delta_{5\mid \delta}  \frac{1}{12}\sigma_1\left(\frac{\delta}{5}\right) +  \frac{4}{21}\sigma_1(4\delta) + \frac{1}{6}\sigma_1(\delta),
\end{multline*}
using $\delta\equiv 3+4\ell\pmod{8}$ in the last step. Since $\sigma_1$ is multiplicative and $\delta$ is odd, plugging in $\sigma_1(4)=7$ and simplifying yields 
\[
\frac{1}{12}\sigma_1(5\delta)+5\delta_{5\mid \delta}  \frac{1}{12}\sigma_1\left(\frac{\delta}{5}\right) +  \frac{4}{3}\sigma_1(\delta) + \frac{1}{6}\sigma_1(\delta)=\frac{1}{12}\sigma_1(5\delta)+5\delta_{5\mid \delta}  \frac{1}{12}\sigma_1\left(\frac{\delta}{5}\right) +  \frac{3}{2}\sigma_1(\delta).
\]
Writing $\delta=5^rm$ with $5\nmid m$ and evaluating $\sigma_1(5^j)=\frac{5^{j+1}-1}{4}$, we compute 
\[
\frac{1}{12}\sigma_1(5\delta)+5\delta_{5\mid \delta}  \frac{1}{12}\sigma_1\left(\frac{\delta}{5}\right)=\frac{1}{2}\sigma_1(\delta).
\]
Hence 
\[
\frac{1}{12}\sigma_1(5\delta)+5\delta_{5\mid \delta}  \frac{1}{12}\sigma_1\left(\frac{\delta}{5}\right) +  \frac{3}{2}\sigma_1(\delta)=2\sigma_1(\delta).
\]
We conclude that the left-hand side of Theorem \ref{thm:PrincipalGenus} with $n=5$ is $2\sigma_1(\delta)$. Comparing with the right-hand side of Theorem \ref{thm:PrincipalGenus-weak} yields
$
c_{\mathcal{Q},5}=\frac{2}{H(2,5)}.
$
By \cite[Proposition 4.1]{Cohen}, we have $H(2,5)=-\frac{2}{5}$, so we conclude that $c_{\mathcal{Q},5}=-5$. We claim that $c_{\mathcal{Q},j}=c_{Q,j}$ for any $Q\in \fg_{-\delta,1,\delta}$, and hence $c_{Q,1}=-1$ and $c_{Q,5}=-5$. Indeed, it is well-known that $\Theta_Q-\Theta_{\mathcal{Q}}$ is a cusp form (e.g. \cite{Walling}). Then $\Theta_{Q,2,m}-\Theta_{\mathcal{Q},2,m}$ is also a cusp form and Theorem \ref{thm:PrincipalGenus-weak} implies that
\begin{align*}
2^{\omega(\delta)} \sum_{d|\delta}d\cdot \left(\mathcal H|V_{d^2} \left(\Theta_{Q,2,m}|V_4-\Theta_{\mathcal{Q},2,m}\right)\right)|U_{\delta} S_{8,j}&=G_{Q,j}-G_{\mathcal{Q},j} \\[-15pt]
&=\left(c_{Q,j}-c_{\mathcal{Q},j}\right)2^{\omega(\delta)}\sigma_1(\delta)\mathcal{H}_{\frac{5}{2}}|S_{8,j}.
\end{align*}
Since $\Theta_{Q,2,j}-\Theta_{\mathcal{Q},2,j}$ is a cusp form and $\widehat{\mathcal{H}}$ grows at most polynomially at the cusps, the left-hand side is a cusp form. Hence if $c_{Q,j}\neq c_{\mathcal{Q},j}$, then $\mathcal{H}_{\frac{5}{2}}|S_{8,j}$ is a cusp form. Since $\mathcal{H}_{\frac{5}{2}}|S_{8,j}$ is not a cusp form by Lemma \ref{lem:HSn8-notcusp}, we obtain that $c_{Q,j}=c_{\mathcal{Q},j}$, as claimed.
\end{proof}

\begin{proof}[Proof of Theorem \ref{theorem: principal genus}]
By Theorem \ref{thm:PrincipalGenus-weak}, for $j\in\{1,5\}$, there exist constants $c_{Q,j}$ such that 
\[
G_{Q,j}=c_{Q,j}2^{\omega(\delta)}\sigma_1(\delta)\mathcal{H}_{\frac{5}{2}}|S_{8,j}.
\]
As shown in the proof of Theorem \ref{thm:PrincipalGenus}, we have $c_{Q,j}=-[j]_8$, yielding the claim.
\end{proof}

\section{Further generalizations}\label{sec:further}
Since Theorems \ref{thm:PrincipalGenus} and \ref{thm:D>1} generalize \eqref{eqn:Eichler}, it is natural to consider class number relations from other theta functions. As noted in the introduction, Theorem \ref{theorem: level reduction} may be used to obtain class number relations involving spherical polynomials. The following example illustrates this phenomenon. 

\begin{Example} Assume the notation and the hypotheses of Theorem \ref{theorem: D>1} and, in particular, that $D>1$. Let $P$ be a spherical polynomial of degree $2$ with respect to a binary quadratic form  $Q(\bm{m})=am_1^2+bm_1m_2+cm_2^2\in \fg_{-\delta,D,N}$ of discriminant $-\delta$, i.e., $P(\bm{m})$ is a linear combination of $(2am_1+bm_2)m_2$ and $m_1(bm_1+2cm_2)$. Then, by \cite[Proposition 2.1]{Shimura-correspondence},
  $$
  \Theta_{Q,P}(\tau):=\sum_{\bm{m}\in\Z^2}P(\bm{m})q^{Q(\bm{m})} \in \sS_{3}(\Gamma_0(\delta),\chi_{-\delta}).
  $$
  Applying Theorem \ref{theorem: level reduction} to $\mathcal H_{D,N}$ and $\Theta_{Q,P}|V_{4}$, we see that
  $$
  \sum_{d|\delta}\left(\mathcal H_{D,N}|U_d
    \cdot \Theta_{Q,P}|V_{4} U_d\right)|U_{\frac{\delta}{d}}
  $$
  is a cusp form in  $\sS^{+}_{\frac{9}{2}}(\Gamma_0(4))$ and hence vanishes.\footnote{The image under the Shimura lift of a cusp form in Kohnen's $+$-space of $\sS_{\frac{9}{2}}(\Gamma_0(4))$ is a cusp form of weight $8$ on $\SL_2(\Z)$, so it must be $0$.} As in Corollary \ref{corollary: D>1}, we find that
  $$
  \sum_{d|\delta}(-1)^{\omega(\gcd(D,d))}d\cdot\left(\mathcal H|V_{d^{2}}
    \cdot\Theta_{Q,P}|V_{4}\right)|U_{\delta}=0.
  $$
  Consequently, for $n\in\N$ with $n\equiv0,1\Pmod 4$, we have
  $$
  \sum_{d|\delta}(-1)^{\omega(\gcd(D,d))}d\sum_{\bm{m}\in\Z^2}
  P(\bm{m})H\left(\frac{\delta n-4Q(\bm{m})}{d^2}\right)=0.
  $$
\end{Example}
\bibliographystyle{plain}


\end{document}